\documentclass[a4paper,12pt,oneside]{amsart}

\usepackage{amsmath,amsfonts,amssymb,amsthm,amsopn,amsxtra}

\newtheorem{Thm}{Theorem}[section]    % theorem like environments
\newtheorem{Lem}[Thm]{Lemma}

\newtheorem{Cor}[Thm]{Corollary}

\theoremstyle{definition}
\newtheorem{varrem}[Thm]{}

\theoremstyle{remark}
\newtheorem{Rem}[Thm]{Remark}
\newtheorem{Rems}[Thm]{Remarks}
\numberwithin{equation}{section}

\DeclareMathOperator{\im}{Im}
\DeclareMathOperator{\re}{Re}
\DeclareMathOperator{\sgn}{sgn}
\DeclareMathOperator{\arccot}{arccot}
\DeclareMathOperator{\hyp}{_2F_1}

\newcommand{\Cal}{\mathcal}
\newcommand{\fr}{\mathfrak}
\newcommand{\R}{\mathbb{R}}        %  real numbers
\newcommand{\C}{\mathbb{C}}        %  complex numbers
\newcommand{\Z}{\mathbb{Z}}        %  integers
\newcommand{\N}{\mathbb{N}}
\newcommand{\T}{\mathbb{T}}

\newcommand{\SL}{\mathit{SL}}
\newcommand{\SU}{\mathit{SU}}
\newcommand{\PSL}{\mathit{PSL}}
\newcommand{\squ}{\mathbin{\scriptstyle\square}}

\begin{document}
\title[Asymptotics of matrix coefficients]{On the asymptotics of matrix
coefficients\vspace{1mm} of representations of $\SL(2,\R)$}
\author{Viktor  Losert}
\address{Institut f\"ur Mathematik, Universit\"at Wien, Strudlhofg.\ 4,
  A 1090 Wien, Austria}
\email{Viktor.Losert@UNIVIE.AC.AT}
\date{22 December 2022}
%--------------------Abstract---------------------------------------------

\begin{abstract}
We study matrix coefficients of the unitary (and also the completely bounded)
representations of $\SL(2,\R)$ and its universal covering group. We describe
the asymptotic distribution of  column vectors in terms of Whittaker
functions, exhibiting also a relationship to the Fourier transform.
\end{abstract}
\maketitle

\baselineskip=1.3\normalbaselineskip
%\setcounter{section}{-1}
%---------------------------Anfang------------------------
For a unitary representation $T$ of a group $G$ on a separable Hilbert space
and an orthonormal basis $(e_m)$\,, the operators $T(g)$ are represented by
(infinite unitary) matrices $(t_{mn}(g))$. For $G=\SL(2,\R)$, standard
matrix coefficients are described by $\fr P_{mn}^\ell(x)$\,, functions on
the real interval $[1,\infty[$ (see \eqref{defP} for details). For
$\ell,m,n$ fixed, it is not hard to give an asymptotic approximation for
$\fr P_{mn}^\ell(x)$ \,(in particular, showing exponential decay of
$\fr P_{mn}^\ell(\cosh2\tau),\fr P_{mn}^\ell(e^\tau)$ for
$\tau\to\infty$), using standard properties of the hypergeometric function.
We review these results in Remark\;\ref{remprinc}\,(a) for the principal
series, with exceptions described in the proof of Theorem\;\ref{apcoprinc},
and in Remark\;\ref{remdisc}\,(a) (discrete series), \ref{di41}
(complementary series). Essentially, this has been
already found out by Bargmann \cite{B}\;\S11. Quite a lot of work
has been done (largely initiated by Harish--Chandra) on generalizations
to semisimple Lie groups (see \cite{Kn}\;Ch.\,VIII,\,XIV), in particular,
investigating the exponential
decay, using the differential equation satisfied by the coefficients and
relating the parameters to data from the Lie algebra. This gave an
important technical tool for many results from the general theory: \
subrepresentation theorem, intertwining operators, Plancherel theorem.

Ehrenpreis, Mautner (\cite{EM2}\;L.\,2.1) show that for any $q\ge1$ there
exists $c>0$ such that
\;$\lvert\fr P_{mn}^\ell(x)\rvert\le
c\;\bigl(\dfrac{\lvert m\rvert+\lvert\ell\rvert+1}{\lvert n\rvert+1}\bigr)^q$
holds for all $1\le x\le q\;,\;\ell\in\C\;,\;m,n\in\Z$ (in fact also for
$m,n\in\R$ with $m-n\in\Z$). In particular, for $\ell,n$ fixed,
$\fr P_{mn}^\ell(x)\to0$ for $\lvert m\rvert\to\infty$\,, uniformly
on bounded subsets of $[1,\infty[$\,.

For general semisimple $G$, Harish--Chandra (\cite{H}\;L.\,17.1) gives an
upper estimate, also for the derivatives of the matrix coefficients of
principal
series representations (formulated in the setting of Eisenstein integrals),
using his $\Xi$-function. The case of $\SL(2,\R)$ has been worked out in
\cite{Bk}\;Th.\,4.1.

In this  paper, we give an approximation of  $\fr P_{mn}^\ell(x)$ by
Whittaker functions (Theorems\;\ref{apcoprinc},\,\ref{apcodisc}). For
$\ell,n$ fixed, $m$ large, it discribes quite well the ``non-small"
part of the functions (see Remark\;\ref{remdisc}\,(c) for comparison to other
methods), giving also rise to some new limit relations. In particular,
one can describe the asymptotic behaviour of
the column vectors of the matrices $(\fr P_{mn}^\ell(x))$ for $x\to\infty$
(Corollary\;\ref{apl2}), exhibiting a connexion with (abelian)
Fourier transforms (Corollary\;\ref{apL2}).

Since it makes no technical difference, I have included
the case of the universal covering group of $\SL(2,\R)$, hence the
results apply to
all connected Lie groups that are locally isomorphic to $\SL(2,\R)$. We
discuss also uniformly bounded representations. There
(Remark\;\ref{remcomp}\,(a)\,), some differences occur
for the universal covering group, compared to the results of \cite{ACD}.
Furthermore (Remark\;\ref{remcomp}\,(b)\,) we make extensions to the listing
of $L^p$-conditions for the coefficients made in \cite{W}.

In \cite{CH},\cite{DH}, Cowling, De Canni\`ere and Haagerup studied
multipliers of the Fourier algebra of certain semisimple Lie groups $G$\,.
When $G$ has finite centre and $K$ is a maximal compact subgroup there is the
special case of $K$--bi-invariant (or ``radial") multipliers. In \cite{CH}
it was shown that these can be characterized by their restriction to a
complementary solvable group $H$ (arising from the Iwasawa decomposition) and
that these multipliers are always completely bounded (with equality of the
norms). For $G=\SL(2,\R)$, I have sketched in \cite{L1} an argument to get
similar results for general multipliers of the Fourier algebra, using
partial Fourier expansion with respect to $K$\,, showing in particular
that all multipliers are completely bounded. This is now the first part
of detailed proofs. It covers basically Prop.\,3 of \cite{L1} (and
related cases).
\section{Basic notions and Main results}\label{Main} % Sec.1
\vskip 1mm%\noindent
\begin{varrem}\label{di11}
Put $G=\SL(2,\R)$ (real $2\!\times\!\!2$\,-matrices of determinant one), let
$K$ be \vspace{1mm}the subgroup of rotations
$k_{\varphi}=
\begin{pmatrix}\cos\varphi & -\sin\varphi\\\sin\varphi& \cos\varphi
\end{pmatrix}$\vspace{.5mm} and $H$ the subgroup of \vspace{-1.5mm}matrices
$g=\begin{pmatrix}a & 0\\b& \frac1a
\end{pmatrix}$
with $a>0,\ b,\varphi\in\R$\,. Recall (part of the Iwasawa decomposition)
that $G=KH$\,, the decomposition of the elements $x=k\,g$ being unique.
$\widetilde G$ shall denote the universal covering group of $G$\,.
Set theoretically and topologically it can be described by its Iwasawa
decomposition $\widetilde G=\widetilde KH$\,, using that $H$ (being simply
connected) embeds as a subgroup into $\widetilde G$\,. We will identify
$\widetilde K$ (the
universal covering group of $K$) with $\R$\,, with covering homomorphism
\,$p(\varphi)=k_\varphi$ \;($\varphi\in\R\,;\linebreak
p\!:\widetilde K\to K$),
extending to \,$p\!:\widetilde G\to G$ \,by \,$p(g)=g$ for $g\in H$\,.
Then the subgroup $Z=\pi\Z$ \,of $\widetilde K$ gives the centre of
$\widetilde G$\,.\vspace{1mm}

The irreducible unitary representations of  $\SL(2,\R)$ have been described
by Barg\-mann (\cite{B}). This was extended to the universal covering group
by Pukanszky (\cite{P}). We use (essentially) the notations (and
parametrization) of Vilenkin (\cite{VK}) and extend it to $\widetilde G$\,,
closely like \cite{Sa}.

Let $\chi=(\ell,\epsilon)$, where
$\ell\in-\frac12+i\,\R\,,\:\epsilon\!\in\!\R$\,.
Then a (strongly continuous) unitary representation $T_\chi$ of
$\widetilde G$ is defined on $\Cal H=L^2(\R)$ (ordinary Lebesgue measure).
This makes up the {\it principal (unitary) series} of representations
(see Remark\,\ref{remrep}\,(a) for more details, references and
generalizations). In
the notation of \cite{Sa}\,Th.\,2.2.1, $T_\chi$ is related to $U_h(\cdot,s)$.
The representations
$T_{(\ell,\epsilon)}\,,\,T_{(\ell,\epsilon+1)}\,,
T_{(-\ell-1,\epsilon)}$ are
unitarily equivalent (\cite{Sa}\,Th.\,2.2.1,\,2.3.2), hence it would be
enough to consider $\ell=\linebreak-\frac12+i\lambda$ with $\lambda\ge0$ and
$-\frac12<\epsilon\le\frac12$\,. In this range, excepting
$\chi=(-\frac12,\frac12)$\,, the representations $T_\chi$ are irreducible
(\cite{Sa}\,Th.\,2.2.2) and mutually non-equivalent (as in \cite{VK}\,6.4.4;
\,$T_{(-\frac12,\frac12)}$ will appear again in the section\,\ref{disc} on
the discrete series).
\\
For $k\in\Z$ \,(hence $k\pi\in Z\subseteq\widetilde K$), one has\vspace{-1mm}
\begin{equation} \label{repcent}
T_\chi(k\pi)=e^{2\pi ik\epsilon}
\end{equation}
Thus $\epsilon$ determines the character induced by
$T_\chi$ on the centre $Z$\,. The values $\epsilon=0,\frac12$ give
representations of $\SL(2,\R)$,
analogously for the other connected Lie groups that are locally isomorphic
to $\SL(2,\R)$, e.g., restricting to $\epsilon=0$ corresponds to
$\PSL(2,\R)\ \,(=\SL(2,\R)/\{k_0,k_\pi\}$; {\it projective special linear
group}).
\\
For $g\in H$\,, $f\in\Cal H\,,x\in\R$\,, one has\vspace{-1mm}
\begin{equation} \label{repaff}
(T_\chi(g)f)(x)=a^{-2\ell}f(a^2x+a\,b)
\end{equation}
This does not depend on $\epsilon$ and corresponds
(up to a character determined by $\ell$) to the standard action of $H$ on $\R$
given by orientation preserving affine linear transformations.

Our main concern will be certain matrix coefficients of the representations.
We use (following \cite{VK}) a fixed orthonormal basis
$(e_m^\chi)_{m\in\Z}$ of the Hilbert space $\Cal H$ \,(see \eqref{defbasis}
below and the Remarks\,\ref{remrep}\,(b)\,). The basis vectors satisfy \;
$T_\chi(\varphi)\,e_m^\chi=e^{2i(m+\epsilon)\varphi}e_m^\chi$ for
$\varphi\in\R\ (=\widetilde K)$ \ (``elliptic basis"). For $g\in\widetilde G$,
we put \ $t_{mn}^\chi(g)=(T_\chi(g)\,e_n^\chi\mid e_m^\chi)$ \
(\ $(\cdot\!\mid\!\cdot)$
denoting the inner product of $\Cal H$). \vspace{.5mm}Thus
$T_\chi(g)\,e_n^\chi=\sum_{m\in\Z}t_{mn}^\chi(g)\,e_m^\chi$\,.
\ Put $A=\Bigl\{\begin{pmatrix}a & 0\\0& \frac1a\end{pmatrix}:
\;a>0\,\Bigr\}\ (<H)$\,.
One has $G=KAK$ \,(resp. $\widetilde G=\widetilde KA\widetilde K$\,;
\,for $\SL(2,\R)$ this is just the classical singular value decomposition).
Hence the functions $t_{mn}^\chi\!:\widetilde G\to\C$ are determined by
their values on $A$\,. We define for \,$\tau\ge0\,,\,m,n\in\Z$\vspace{-1mm}
\begin{equation} \label{defP}		%  (1.3)  Koeffizienten
\fr P_{m+\epsilon\:n+\epsilon}^\ell(\cosh2\tau)=t_{mn}^\chi
\Bigl(\begin{pmatrix}e^\tau & 0 \\  0 & e^{-\tau}
\end{pmatrix} \Bigr)
\end{equation}
(\cite{VK}\,6.5.2\,(4), see also p.\,314). Up to the constant factor
$e^{-2i\tau\lambda}$ arising in the formula \eqref{repaff} for $T_\chi(g)$
when $g=\bigl(\begin{smallmatrix}e^\tau & 0 \\ \!0 & e^{-\tau}
\end{smallmatrix}\bigr)$, this gives matrix representations for the
operator $f\mapsto e^\tau f(e^{2\tau}\,\cdot)$ on $L^2(\R)$ with respect to
the family of bases $(e_m^\chi)$. We take\vspace{-2mm}
\begin{equation} \label{defbasis}
e_m^\chi(t)=
\dfrac1{\sqrt{\pi}}\,e^{-i\pi(m+\epsilon)}\,e^{2i(m+\epsilon)\arctan(t)}
(t^2+1)^\ell\qquad (m\in\Z).
\end{equation}
For $t\in\R$ one has $e^{-i\pi(m+\epsilon)}\,e^{2i(m+\epsilon)\arctan(t)}=
(t-i)^{m+\epsilon}(t+i)^{-m-\epsilon}$  (we always use the standard
branches for complex powers).
For $\epsilon=0$ (``first principal series" of representations of $\SL(2,\R)$)
this reduces to the formula given in \cite{L1}\,p.\,12.
The definition of $\fr P$ integrates $\epsilon$ by shifting the index set to
$\epsilon+\Z$\,. This includes the equivalence of $T_{(\ell,\epsilon)}$ and
$T_{(\ell,\epsilon+1)}$\,, since
$e_m^{(\ell,\epsilon)}=e_{m-1}^{(\ell,\epsilon+1)}$ implies
\,$t_{mn}^{(\ell,\epsilon)}=t_{m-1\,n-1}^{(\ell,\epsilon+1)}$\,.
For example, $\epsilon=\frac12$ gives the ``second principal series" of
representations of $\SL(2,\R)$ with coefficients $\fr P_{mn}^\ell$ where
$m,n$ are half-integers ($m,n\in\frac12+\Z$).\\[1mm]
Upon first reading, one may skip the next remarks (with technical details)
and pass to Theorem\;\ref{apcoprinc}.\vspace*{-2.5mm}
\begin{Rems}	  \label{remrep}
(a) {\it Details on the representations:}
\;Traditionally (following Barg\-mann), many authors first define the
representations for the isomorphic group
$\SU(1,1)=\{\bigl(\begin{smallmatrix}\alpha & \bar\beta \\ \beta & \bar\alpha
\end{smallmatrix}\bigr)\!\in\SL(2,\C)\}$ \,(and its universal covering group
$\widetilde{\SU}$).
In this version, $T_\chi$ (or more precisely $T_\chi^{\T}$) can be defined
for arbitrary $\chi=(\ell,\epsilon)\in\C^2$ as
a strongly continuous representation of $\widetilde{\SU}$ by bounded operators
on $L^2(\T)$
\,(where $\T=\{z\in\C:\lvert z\rvert=1\}$ with normalized Haar measure).
Infinitesimally, it corresponds to the Harish-Chandra module $U(\nu^+,\nu^-)$
of the Lie algebra (notation of \cite{HT}\,p.\,61,\,(1.2.5), with
$\nu^+=-\ell+\epsilon\,,\ \nu^-=-\ell-\epsilon$\,), see \cite{VK}\,p.\,300
and \cite{Sa}\,p.\,38. In the notation of \cite{Sa}, $T_\chi=U_h(\cdot,s)$
with $s=-\ell-\frac12\,,\,h=\epsilon$\,. By \cite{Sa}\,Th.\,2.2.1,
$T_\chi$ depends analytically on $\ell,\epsilon$\,. In the notation of
\cite{P}\,p.\,102, the unitary case (i.e. $\re(\ell)=-\frac12\,,\;
\epsilon\in\R$) is denoted as $C_q^{(\tau)}$ with
$q=\lvert\ell\rvert^2>\frac14\,,\;\tau=\epsilon$ (the case $q=\frac14$
\,i.e.\!
$\ell=-\frac12$ is shifted there to the complementary and the discrete series).
See Remark\;\ref{remcomp}\,(c) for further properties.
For $\epsilon\in\frac12\,\Z$ one gets representations of $\SU(1,1)$.
\\
Explicit formulas for the case of $\SU(1,1)$ are given in
\cite{VK}\,p.\,299,\,(4) \,(writing $\tau$ instead of $\ell$),
\,\cite{Sa}\,p.\,12,\,(1.1.2) and then for $\widetilde{\SU}$ in
\cite{VK}\,p.\,310,\,(1) and \,\cite{Sa}\,p.\,35, (2.2.9). But for $\SL(2,\R)$
quite a number of variations are used by some authors. Observe that
\cite{VK} and \cite{Sa} apply different isomorphisms for the transfer from
$\SU(1,1)$ to
$\SL(2,\R)$ (and also for the corresponding transformations from $L^2(\T)$
to Hilbert spaces on $\R$). As a result, the behaviour on the centre
(see \eqref{repcent}
above) becomes different in the version of \cite{Sa} (for $h=0,\frac12$ he
writes $T_h(\cdot,s)$
for the corresponding representations of $\SL(2,\R)$ on $L^2(\R)$ --
\cite{Sa}\,p.\,12,\,(1.1.4); but the generalization to $\widetilde G$ is
not elaborated in full detail; the different choice of the isomorphism implies
that in the case of $\widetilde G$ the correspondence becomes
$h=-\epsilon$\,, see e.g. \cite{Sa}\,L.\,3.1.3,\,p.\,59).
\\[0mm plus 1mm]
A direct construction for $\SL(2,\R)$ (similar to that in \cite{VK}\,6.4.1)
has been given in \cite{ACD}\;sec.\,3. This extends easily to $\widetilde G$
(again,
we stick to the style of \cite{VK})\,: \  Let (in slightly misleading
notation) \,$\widetilde{\R^2}$ be the universal covering space of
$\R^2\setminus\{0\}$. This can be written as
\,$\widetilde{\R^2}=\R\times]0,\infty[$ with covering map
$q(\omega,t)=(t\,\cos\omega,t\,\sin\omega)$\linebreak
($q\!:\widetilde{\R^2}\to\R^2\setminus\{0\}$; polar coordinates). Let
$\widetilde{\Cal D}_\chi$ be the space of $C^\infty$-functions
$F\!:\widetilde{\R^2}\to\C$ such that $F(\omega,a\,t)=a^{2\ell}F(\omega,t)$
and $F(\omega+\pi,t)=e^{-2\pi i\epsilon}F(\omega,t)$ for
$\omega\in\R\,,\; a,t>0$\,. Considering $\R^2$ as {\it row} vectors, $G$ acts
on $\R^2\setminus\{0\}$ by {\it right} multiplication ($\cdot$) and this lifts
to a right action of $\widetilde G$ on $\widetilde{\R^2}$\,, denoted by
$\zeta\squ g$\,, i.e., $q(\zeta\squ g)=q(\zeta)\cdot p(g)$ for
$\zeta\in\widetilde{\R^2}\,,\,g\in\widetilde G$\,.
Observe that our parametrization of $K$ implies that
$(\omega,t)\squ\varphi=(\omega-\varphi,t)$ for $\varphi\in\widetilde K$
and that the (half-)\,sheets $[k\pi,(k+1)\pi[\,\times\,]0,\infty[\ \;(k\in\Z)$ are
$H$-invariant.
\\[0mm plus 1mm]
Then \;$(T_\chi^s(g)F)(\zeta)=F(\zeta\squ g)$ \,defines a representation of
$\widetilde G$ on $\widetilde{\Cal D}_\chi$ (``smooth version"),
satisfying \eqref{repcent} on the centre. $\widetilde{\Cal D}_\chi$ and
$T_\chi^s$ depend just on the coset of $\epsilon\!\!\mod1$.
The embedding of $\R$ into
$\R^2$ used in \cite{VK}\;p.\,376 (which is
different from that in \cite{ACD}\;p.\,132) lifts
to an embedding into $\widetilde{\R^2}$ given
by \,$x\mapsto(\arccot x,\sqrt{x^2+1})$, i.e.,
$q(\arccot x,\sqrt{x^2+1})=(x,1)$
\,(we use the principal branch of $\arccot$, i.e., $\arccot x\in\,]0,\pi[$\,).
Considering in this way $\R$ as a subset of $\widetilde{\R^2}$, it is easy
to see that functions $F\in\widetilde{\Cal D}_\chi$ are determined uniquely
by their restriction to $\R$ \,and one obtains a (isomorphic) space
$\Cal D_\chi^\R$ of $C^\infty$-functions on $\R$ \,(see
Remark\;\ref{remcomp}\,(c) for further details).
\\[-.5mm]
Explicitly, this gives for $f\in\Cal D_\chi^\R\,,\ g\in\widetilde G$ with
\vspace{-1mm}
$p(g)=\begin{pmatrix}a & b\\c& d\end{pmatrix} $ and
$x\in\R$ with $bx+d\neq0$, \ \;
$(T_\chi^s(g)f)(x)=
e^{2k\pi i\epsilon}\,\lvert bx+d\rvert^{2\ell}
f\bigl(\dfrac{ax+c}{bx+d}\bigr)$, \,where \vspace{1mm}
$k\in\Z$ is determined by $(\arccot x,\sqrt{x^2+1})\squ g=(\omega-k\pi,\dots)$
with $\omega\in\,]0,\pi[$\,. For $g\in H$ (where $k=0$), one obtains
$\eqref{repaff}$.
For $\ell=-\frac12+i\lambda\,,\ \epsilon,\lambda\in\R$\,, $\Cal D_\chi$ is
a dense subset of $L^2(\R)$ and $T_\chi^s$ extends
to a unitary representation $T_\chi$ (or more precisely $T_\chi^{\R}$) of
$\widetilde G$ on $L^2(\R)$.
For $\epsilon=0,\frac12$ (observe that $k\equiv\frac{1-sgn(bx+d)}2\!\mod2$)
this coincides %\pagebreak
with \cite{VK}\,7.1.2\,(3)
(denoted there as $\hat T_\chi$), \cite{Sa}\,p.\,12,\,(1.1.4) and
\cite{KS}\;p.\,3 (with $s=\ell+1$).
There is also a ``compact model" for the representations,
obtained by restricting $F$ to $[0,\pi]\times\{1\}$\,, but the explicit
formulas are not as nice as in the case of $\SU(1,1)$.
\\[0mm plus.5mm]
Basically, this is the induced representation from the subgroup $HZ$ of
$\widetilde G$\,, coming from the character
$\psi(g)=a^{2\ell+1},\;\psi(k\pi)=e^{2k\pi i\epsilon}\ (g\in H\,,\,k\pi\in Z)$,
when using {\it right} translations of $\widetilde G$ and right cosets
(as done originally by Mackey). If $H$ is replaced by upper triangular
matrices, this corresponds to the embedding of $\R$ used in \cite{ACD}.
Similarly, one can use the action of
$G$ on $\R^2$ by left multiplication which corresponds to the realization
of induced representations by left translations. This produces somewhat
different expressions for the representations and slight deviations in the
coefficients.
\vskip.5mm plus 1.5mm
\item[(b)] {\it On the coefficient functions:}
In \cite{VK}\;6.5.2 the functions $\fr P_{m+\epsilon\,n+\epsilon}^\ell$
are defined for $\SU(1,1)$, using the basis $(z^{-m})_{m\in\Z}$ of
$L^2(\T)$ \,(and this works for arbitrary $\ell,\epsilon\in\C$). As mentioned
in (a), various transformations are used in the
literature to get to $\SL(2,\R)$. We use here
$h=\begin{pmatrix}1 & 1\\i& -i\end{pmatrix}$ which defines an isomorphism
$g\mapsto hgh^{-1}$ from $\SU(1,1)$ to $\SL(2,\R)$ (and then also for the
covering groups) and transforms the one-parameter group used for the
definition of $\fr P$ in the way leading to \eqref{defP}.
For a function $f\!:\T\to\C$ put
$(E_\chi f)(x)=
\frac1{\sqrt\pi}\,e^{-2i\epsilon\arccot x}(x^2+1)^\ell
f(\frac{x+i}{x-i})\ \ (\text{giving }\linebreak E_\chi f\!:\R\to\C$\,). For
$\ell=-\frac12+i\lambda\,,\ \epsilon,\lambda\in\R$\,,
$E_\chi$ defines an
isometric isomorphism of $L^2(\T)$ and $L^2(\R)$, linking the
representations $T_\chi$ of $\SU(1,1)$ (and its universal covering group)
as defined (coherently) in \cite{VK} and \cite{Sa} with those of
$\widetilde G$  described in (a), 
i.e.,
$T_\chi^{\R}(hgh^{-1})=E_\chi\,T_\chi^{\T}(g)\,E_\chi^{-1}$
(for $g\in\widetilde{\SU}$).
It gives the basis of \eqref{defbasis},
i.e., $e_m^\chi=E_\chi e_m$\,, where $e_m(z)=z^{-m}$, hence \eqref{defP}
follows
\ (observe that the basis $(\psi_{n\chi})$ appearing in (1) of
\cite{VK}\,7.1.3 is not the one that arises from the procedure in
\cite{VK}\,7.1.2 and that it does not exactly produce the functions
$\fr P_{m+\epsilon\,n+\epsilon}^\ell$).
For general $\chi\in\C^2,\ E_\chi$ induces a Hilbert space structure on
its image $E_\chi L^2(\T)$, and as above this defines a representation $T_\chi$
(or more precisely $T_\chi^{\R}$) of $\widetilde G$ by bounded operators
on a Hilbert space $\Cal H_\chi$ on $\R$\,, containing $\Cal D_\chi^\R$\,.
Again, $e_m^\chi\ (=E_\chi e_m)$
of \eqref{defbasis} gives an orthonormal basis of $\Cal H_\chi$ \,and
$T_\chi^{\R}$ extends $T_\chi^s$ \,(explicit formula in (a)).
$T_\chi,\Cal H_\chi$ depend only on the coset of \,$\epsilon\!\!\mod1$
(but $E_\chi\,,e_m^\chi$ depend on $\epsilon$), in
contrast to the models on $\T$ where
$T_{(\ell,\epsilon)}^{\T}\neq T_{(\ell,\epsilon')}^{\T}$
for $\epsilon\neq\epsilon'$, hence
$T_{(\ell,\epsilon)}^{\T},T_{(\ell,\epsilon')}^{\T}$ are just unitarily
equivalent for $\epsilon-\epsilon'\in\Z$ \,(see Remark\;\ref{remcomp}\,(c)
for further properties).
\end{Rems}
We will show now that for large $m$ or $x$\,,
$\fr P_{mn}^\ell(x)$ can be approximated quite well (already for
moderately large values) using
the Whittaker function (see the next subsection for definition).
Basically, this corresponds to a  ``confluent limit" for the hypergeometric
function $\hyp$\,, compare \cite{Lu}\;Sec.\,III.3.5, \cite{VK}\;3.5.2\,(14).
We give also an error estimate, showing some kind %\pagebreak
of uniformity for fixed $n$ and $\ell$\,. The proof will use the integral
representation of $\hyp$ by Whittaker functions (see \eqref{averwhit}).
\newpage\noindent
We write $\N_0=\{0,1,\cdots\}$.
\end{varrem}
%		  Th.1.3
\begin{Thm} \label{apcoprinc} %[Approximation of coefficients]
For \,$n\in\R$, \,$\ell=-\frac12+i\lambda$ ($\lambda\in\R$) fixed,
\begin{multline*}
\sup\,\biggl\{\ \Bigl\lvert\,\fr P_{mn}^\ell(x)\,-\,
\frac{(-1)^{m-n}}{m^{\ell+1}\Gamma(n-\ell)}\,
W_{n,\,i\lambda}\Bigl(\frac{2\,m}x\Bigr)\,\Bigr\rvert\ \:
x^{\textstyle\frac12}\,m^2\ :\\[-2mm]
x\ge1,\,m\in n+\N_0\,,\,m>0\,\biggr\}\quad
\text{is finite.}
\end{multline*}\vspace{-9mm}
\end{Thm}\noindent
In the special case $\ell=-\frac12\,,\,n\in\frac12-\N$ (arising for
$\chi=(-\frac12,\frac12)$) one uses the analytic extension
of $\frac1{\Gamma(t)}$ by $0$ for $t=0,-1,\dots$
\\%[1.5mm]
Thus $\frac{(-1)^{m-n}}{m^{\ell+1}\Gamma(n-\ell)}\,
W_{n,\,i\lambda}\bigl(\frac{2\,m}x\bigr)$ gives an approximation for
$\fr P_{mn}^\ell(x)$ with error bound $\frac c{\sqrt{x}\,m^2}$\,,
uniformly on the $(x,m)$-domain above.\vspace{.5mm}
In particular,
$\sqrt{m}\,\fr P_{mn}^\ell(m\,x)-
\frac{(-1)^{m-n}}{m^{i\,\lambda}\Gamma(n-\ell)}\,
W_{n,\,i\lambda}(\frac2x)\to0$ \,for $m\to\infty\ (m\in n+\Z)$,
compare with Remark\,\ref{remprinc}\,(a). See below for $m<0$.
\\[2mm]
This supplies an asymptotic expression for the column vectors of the
matrices $(\fr P_{mn}^\ell(x))$ \;(for the $l^2$-norm).
\vspace{-3mm}
\begin{Cor} \label{apl2} % Cor.1.4
For $n,\ell$ fixed as in the Theorem, one has
\begin{align*}
\biggl\lVert\ \Bigl(\fr P_{mn}^\ell(x)\,-\,
&\frac{(-1)^{m-n}}{m^{\ell+1}\Gamma(n-\ell)}\,
W_{n,\,i\lambda}\Bigl(\frac{2\,m}x\Bigr)\,\Bigr)_{m>0}\;\biggr\rVert_2\
\to\ 0 \qquad\text{and}
\\[.5mm]
\biggl\lVert\ \Bigl(\fr P_{mn}^\ell(\frac{x+\frac1x}2)\,-\,
&\frac{(-1)^{m-n}}{m^{\ell+1}\Gamma(n-\ell)}\,
W_{n,\,i\lambda}\Bigl(\frac{4\,m}x\Bigr)\,\Bigr)_{m>0}\;\biggr\rVert_2\
\to\ 0 \qquad
\text{for \ }x\to\infty\,.\vspace{2mm}
\end{align*}
\end{Cor}\noindent
Since $\fr P_{mn}^\ell=\fr P_{-m\,-n}^\ell$ \,(\cite{VK}\,6.5.5\,(1)),
Theorem\;\ref{apcoprinc} easily provides an approximation of 
$\fr P_{mn}^\ell(x)$ by 
$\frac{(-1)^{m-n}}{{\lvert m\rvert}^{\ell+1}\Gamma(-n-\ell)}\,
W_{-n,\,i\lambda}\bigl(\frac{2\,\lvert m\rvert}x\bigr)$ for $m<0$
with $m\in n-\N_0$\,.\vspace{.2mm}
This gives an approximation for the ``lower half"
$\bigl(\,\fr P_{mn}^\ell(x)\bigl)_{m<0}$ in terms of $W_{-n,\,i\lambda}$
(showing some asymmetry when $n\neq0$) .
Since \,$\fr P_{mn}^\ell=(-1)^{m-n}\,\fr P_{nm}^{-\ell-1}$
(\cite{VK}\,p.\,319\,(6)),
similar approximations are obtained for the row vectors (see \eqref{th1mod}
below). Recall that by \eqref{defP}, \,$\fr P_{mn}^\ell(\frac{x+\frac1x}2)$
\;(where $x\ge1$\,, \,$m,\!n\in\epsilon+\Z$) \,gives the matrix coefficients
of the \vspace{.5mm}operator
\,$T_\chi \Bigl(\begin{pmatrix}\sqrt x & 0 \\  0 & \frac1{\sqrt x}
\end{pmatrix} \Bigr)$
with respect to the basis $(e_{m-\epsilon}^\chi)$\,. The second
form of Corollary\,\ref{apl2} corresponds to the statement in
\cite{L1}\,Prop.\,3.\vspace{1mm}

The limit arising in Corollary\;\ref{apl2} is closely related to the
{\it Fourier transform} of~$e_{n-\epsilon}^\chi$ (on $\R$).
For $f\in L^1(\R)$ we take the version
$\hat f(y)=\dfrac1{\sqrt{2\pi}}\,\int\limits_{\R}e^{-ity}f(t)\,dt$
\,($y\in\R$) and
then the same notation will be used for the extension to $L^2$ and to
tempered distributions. From \cite{E}\,p.119\,(12) (the sign in the formula
for $y>0$ is wrong) or \cite{VK}\,p.\,444\,(15$'$) we get for
$n\in\epsilon+\Z$ and $\chi=(\ell,\epsilon)$, $y\neq0$
\begin{equation} \label{fourbasis}
\widehat{e_{n-\epsilon}^\chi}(y)\ =\ e^{-i\pi n}\,
\frac{2^{\ell+\frac12}\,\lvert y\rvert^{-\ell-1}}
{\Gamma(\sgn(y)\mspace{2mu}n-\ell)}
\;W_{\sgn(y)n,\,\ell+\frac12}(2\lvert y\rvert)\ .
\end{equation}
We use again the extension of $\frac1\Gamma$\,.
For $\ell,\epsilon\in\C$ with $\re(\ell)<-\frac12$\,, this is the classical
Fourier transform. Then, by analyticity, one gets the Plancherel transform
($L^2$-transform) for $\re(\ell)<-\frac14$\,, and for %arbitrary $\ell\in\C$
$\re(\ell)<0$
this holds in the distributional sense (Fourier-Schwartz transform; in fact
the right-hand side is integrable in that case). For $\re(\ell)\ge0$ some
modifications are necessary.

It follows that in a certain sense, the $n$-th column vector of the
matrix of \linebreak $T_\chi \Bigl(\begin{pmatrix}x & 0 \\  0 & \frac1x
\end{pmatrix} \Bigr)$ approaches $\widehat{e_{n-\epsilon}^\chi}$ for
$x\to\infty$ \,(see also Corollary\,\ref{apL2} below).
Hence, somehow the representations of $G$ (resp.\,$\widetilde G$\,)
are asymptotically recovered in their Whittaker model. But as
far as I see, this does not fit into classical
concepts like weak convergence of representations. For this reason I was
passing to ultra products and singular states of the von Neumann algebra
in \cite{L1}. A typical feature is given in the
next corollary.
\begin{Cor} \label{apfour}	  % Cor.1.5
Let $h\!:\R\to\C$ be continuous with compact support. Take
$\chi=(-\frac12+i\lambda,\epsilon)$ with $\lambda,\epsilon\in\R$ and for
$x\ge1$ put
\,$h_x=\dfrac{\sqrt2}{x^{1+2i\lambda}}
\sum\limits_{m\in\epsilon+\Z}e^{i\pi m}h\bigl(\dfrac{2m}{x^2}\bigr)
\,e_{m-\epsilon}^\chi$\,.
Then for any $g\in L^2(\R)$ we get that
\[\bigl(\:T_\chi \Bigl(\begin{pmatrix}x & 0 \\  0 & \frac1x
\end{pmatrix} \Bigr)g\mid\, h_x\,\bigr)\ \to\ (\hat g\mid h)\quad
\text{ for }\quad x\to\infty\,.
\]
\end{Cor}\noindent
By unitarity, this means that
$T_\chi \Bigl(\begin{pmatrix}\frac1x & 0 \\  0 & x\end{pmatrix} \Bigr)h_x
\to\hat h$ weakly in $L^2(\R)$ for $x\to\infty$\,. Similarly, one can show
that
$T_\chi \Bigl(\begin{pmatrix}x & 0 \\  0 & \frac1x\end{pmatrix} \Bigr)h_x
\to0$ weakly for $x\to\infty$\,.
\\[5mm]
$c_M$ shall denote the indicator function of a set $M$\,.\vspace{-3mm}
\begin{Cor} \label{apL2}	%  Cor.1.6
For $n,\ell$ fixed as in the Theorem, $\chi=(\ell,\epsilon)$,
with $n\in\epsilon+\Z$\,, one has\vspace{-2mm}
\[\frac {(2x)^{\ell+1}}{\sqrt 2}\sum_{m\in\epsilon+\Z}
e^{-i\pi m}\,\fr P_{mn}^\ell(x)\:
c_{[\frac mx,\frac{m+1}x[}\ \to\ \widehat{e_{n-\epsilon}^\chi}\quad
\text{ for \ }x\to\infty
\]
in $L^2(\R)$ \;(with respect to $\lVert\cdot\rVert_2$).\vspace{-1mm}
\end{Cor} \noindent
A similar argument shows pointwise convergence on $\R\setminus\{0\}$\,.
Also,
\\
$\lim_{x\to\infty}\frac {x^{2\ell+2}}{\sqrt 2}\sum_{m\in\Z}
e^{-i\pi(m+\epsilon)}\,
\bigl(\:T_\chi \Bigl(\begin{pmatrix}x & 0 \\  0 & \frac1x
\end{pmatrix} \Bigr)\,g\mid\, e_m^\chi\,\bigr)\;
c_{[\frac{2m}{x^2},\frac{2(m+1)}{x^2}[}\ =\ \hat g$
\;in $L^2(\R)$ \;(with respect to $\lVert\cdot\rVert_2$) holds for
each $g\in L^2(\R)$\,.
\begin{Rems}	  \label{remprinc}	% Rem.1.7
(a) \ For $\ell$ as above with $\lambda\neq0$ and $m,n\in\R$ with $m-n\in\Z$
it follows e.g. from \cite{VK}\,p.\,396(3) (which has a slight misprint) or
\cite{DLMF}\,Eq.\,15.8.3 that
\;$\sqrt x\:\fr P_{mn}^\ell(x)-\alpha\,\cos(\lambda\log(x)+\beta)\to0$ for
$x\to\infty$ with $\alpha\in\C\,,\,\beta\in\R$ depending on $\ell,m,n$\,;
this gives the ``constant term for the Harish-Chandra expansion"
\cite{Kn}\;Th.\,14.6
\,(for $m-n\ge0$, using Pochhammer's symbol, one has \
$\alpha=\linebreak\sgn((\ell-m+1)_{m-n})\:
\frac{2\,\lvert\cos(\pi(\epsilon-i\lambda))\rvert}
{\sqrt{\pi\lambda\sinh(2\pi\lambda)}}$\,; for $m=n=0$
[zonal spherical functions] this can also be expressed using Harish-Chandra's
$c$-function \cite{Kn}\,p.\,279, compare \cite{VK}\,p.\,402(6)\;).
See also Remark\;\ref{remdisc} for refinements and a comparison.
\\
Thus (for $\ell\neq-\frac12$)
the individual entries
of the matrix $(\fr P_{mn}^\ell(x))$ tend to $0$ with speed
$\frac1{\sqrt x}$ \,for $x\to\infty$\,; see also the proof for the exceptions
in section\,\ref{proofprinc} and Remark\,\ref{remdisc}\,(a) for the case
$\ell=-\frac12$\,. But by unitarity, the row vectors must have length 1
(in $l^2(\epsilon+\Z)$\,) for each $x$\,. Informally speaking, for fixed $n$
the  ``main content" of the $n$-the column moves to $\pm\infty$
proportionally to~$\pm x$
when $x\to\infty$ and there is a ``limiting distribution" (in $l^2$-sense)
which is closely related to $\widehat{e_{n-\epsilon}^\chi}$\,.
\\[2mm]
(b) \ For arbitrary $\ell\in\C$ put
\,$\ell_\R=\re(\ell)\,,\;\ell_1=\frac12-\lvert\ell_\R+\frac12\rvert$.
Then for $n\in\R$ fixed, the same type of argument gives an estimate
\;$\fr P_{mn}^\ell(x)\,=\,
\dfrac{(-1)^{m-n}}{m^{\ell+1}\Gamma(n-\ell)}\,
W_{n,\,\ell+\frac12}\Bigl(\dfrac{2\,m}x\Bigr)+
O(\frac1{m^{\ell_\R+3}}\,(\frac mx)^{\ell_1})$\vspace{1mm}
uniformly for all $x\ge1$ and all %\linebreak
$m\in n+\N_0$ with $m>\max(0,-\ell_1)$.
\end{Rems}
\vskip 2mm
Corresponding results hold for the representations of the discrete
series, see section\,\ref{disc} for definitions and notation.
\vspace{-3mm}
%		       Th.1.8
\begin{Thm} \label{apcodisc} %[Approximation of coefficients discrete]
For \,$\ell,n\in\R$ fixed with $\ell\le-\frac12\,,\ n\in-\ell+\N_0$\,,
\begin{multline*}
\sup\,\biggl\{\ \Bigl\lvert\,\Cal P_{mn}^\ell(x)\,-\,
\frac{(-1)^{m-n}}{\bigl(m\,(\ell+n)!\,\Gamma(n-\ell)\bigr)^{\frac12}}\,
W_{n,\,\ell+\frac12}\Bigl(\frac{2\,m}x\Bigr)\,\Bigr\rvert\
x^{\textstyle\frac12}\,m^2\ :\\[-2mm]
x\ge1,\,m\in n+\N_0\,\biggr\}\quad
\text{is finite.}
\end{multline*}
\end{Thm}\noindent
By \eqref{defPd-}, one has
$\Cal P_{mn}^\ell=\Cal P_{-m\,-n}^\ell$\,. Hence there is a corresponding
approximation in the case $\ell\ge n\ge m$\,.
Similarly for the complementary series and the uniformly bounded
representations (see\,\ref{di42}).\vspace{1mm}
\begin{varrem}\label{di14}		% 1.9
{\bf Whittaker functions, hypergeometric function,\newline
incomplete gamma function}
\\
The Whittaker functions are defined by
\[
W_{\nu\!,\rho}(z)=
\frac{z^{\rho+\frac12}\,e^{-\frac z2}}{\Gamma(\rho-\nu+\frac12)}\,
\int\limits_0^\infty e^{-zu}\,u ^{\rho-\nu-\frac12}\,
(1+u)^{\rho+\nu-\frac12}\,du
\]
for \,$\re z>0$, $\re(\rho-\nu+\frac12)>0$. By analytic continuation
one gets extensions for all $\nu,\rho\in\C$ and for $z$ in the cut plane
$\C\setminus]-\infty,0]$\,. They are analytic in $\nu,\rho,z$\,. See
\cite{VK}\,5.2 for further properties.
\\
The (Gaussian) hypergeometric function will be denoted by $\hyp$\,. See
\cite{VK}\,3.5.3 for definition and some basic properties. It will be
used to express $\fr P_{mn}^\ell$\,. By \cite{VK}\,p.\,313\,(2)
\[\fr P_{mn}^\ell(x)=\binom{\ell-n}{m-n}
\Bigl(\frac{x-1}2\Bigr)^{\!\frac{m-n}2}
\Bigl(\frac{x+1}2\Bigr)^{\!-\frac{m+n}2}\,
\hyp\bigl(-\ell-n,\ell+1-n;m-n+1;\frac{1-x}2\bigr)
\]
holds for \,$m-n\in\N_0\,,\,x\ge1$ \,(and this would provide also an analytic
extension of $\fr P_{mn}^\ell(x)$ to a cut plane $\C\setminus]-\infty,1[$\,).
Using again that
$\fr P_{nm}^\ell=\fr P_{\!-n\,-m}^\ell$ (see above) and combining with
\cite{DLMF}\,Eq.\,13.23.4 it follows that
\begin{multline}    \label{averwhit} % representation as an average
\fr P_{nm}^\ell(x)=\frac1{\Gamma(\ell+n+1)}\,
\frac{\Gamma(m)}{\Gamma(m-\ell)}
\Bigl(1-\frac1{x^2}\Bigr)^{\!\frac m2}\Bigl(\frac{x+1}{x-1}\Bigr)^{\!\frac n2}
\cdot
\\
\frac1{\Gamma(m)}
\int_0^\infty e^{-t}\,t^{m-1}\,W_{\!n,\,\ell+\frac12}(\tfrac{2t}x)\,dt
\end{multline}
for all \,$x,\ell,m,n\in\C$ such that
\,$\re(x)>1\,,\,m-n\in\N_0$ and
$\re(m)+\frac12>\lvert\re(\ell+\frac12)\rvert$\linebreak
\,(in fact, the substitution of $\hyp$ would be possible for $\re(x)>0\,,\,
\re(m)+\frac12>\lvert\re(\ell+\frac12)\rvert$ -- to ensure convergence of the
integral; for $m=0$ one has to cancel $\Gamma(m)$\,). The second part of this
expression is (restricting now to $m\in\R\,,\linebreak
m>0$) an average of
$W_{\!n,\,\ell+\frac12}$ with respect to a gamma (probability)
distribution. To get estimates, we use results on the asymptotics of
incomplete gamma functions. For $a,x>0$ (real) put
$P(a,x)=\frac1{\Gamma(a)}\,\int_0^x t^{a-1}e^{-t}\,dt\,,\
Q(a,x)=1-P(a,x)=\frac1{\Gamma(a)}\,\int_x^\infty t^{a-1}e^{-t}\,dt$\,.
In \cite{Te}\,Temme gives an algorithm to compute the coefficients for
asymptotic expansions of $P,Q$ and derives uniform bounds for the remainders.
\end{varrem}
\begin{Lem} \label{incgam} % Asymptotik Temme  L.1.10
Put $\lambda=\dfrac xa$\,. Then\vspace{-1mm}
\[
P(a,x)=\frac{(\lambda\,e^{1-\lambda})^a}{\sqrt{2\pi a}}\Bigl(\frac1{1-\lambda}
+O\bigl(\frac1{(1-\lambda)^3\,a}\bigr)\,\Bigr)\quad\text{holds uniformly for }
0<x<a\,,
\]\vspace{-2mm}
\[
Q(a,x)=\frac{(\lambda\,e^{1-\lambda})^a}{\sqrt{2\pi a}}\Bigl(\frac1{\lambda-1}
+O\bigl(\frac1{(\lambda-1)^3\,a}+\frac1a\bigr)\,\Bigr)\quad\text{holds
uniformly for }0<a<x\,.\vspace{1mm}
\]
\end{Lem}\noindent
Presumably, one may take $2.1$ as a constant in the $O$-estimate.
\begin{proof}
This is the first term ($k=0$) of the expansion \cite{Te}\,(1.5) combined
with the first term in \cite{DLMF}\,Eq.\,7.12.1 of the expansion for the
complementary error function. The uniform error bound follows from the
formulas for $Q_1,P_1$ in \cite{Te}\,(2.14),(2.16).
\end{proof}
%%%%%%%%%%%%%%%%%%%%%%%
\section{Proof of theorem\,\ref{apcoprinc} and corollaries}
\label{proofprinc}		% Sec.2
\begin{Lem}	 \label{intrem} %Integralabschätzung Restbereich     L.2.1
Let $\alpha\ge0$ (real) be fixed and assume that \,$f\!:\,]0,\infty[\to\C$ is
a measurable function such that $f(t)\,t^{-\alpha}$ is bounded on
$]0,\infty[$. Then for any fixed $\beta>0$ we have
\[
\Bigl\lvert\,\int_0^{\frac m2}e^{-t}\,t^{m-1}f(y\,t)\,dt\,\Bigr\rvert+
\Bigl\lvert\,\int_{2m}^\infty e^{-t}\,t^{m-1}f(y\,t)\,dt\,\Bigr\rvert=
O(\,y^\alpha (m+1)^{-\beta}\,\Gamma(m+\alpha)\,)
\]
uniformly for $m,y>0$.
\end{Lem}
\begin{proof}
Assume that $\lvert f(t)\rvert\le c\,t^\alpha$ for all $t>0$\,. Then
$\lvert\,\int_0^{\frac m2}e^{-t}\,t^{m-1}f(y\,t)\,dt\,\rvert\le
c\,y^\alpha\,P(m+\alpha,\frac m2)\,\Gamma(m+\alpha)$. 
Using that $\lambda=\frac m{2(m+\alpha)}\le\frac12$\,, the estimate
$P(m+\alpha,\frac m2)=O((m+1)^{-\beta})$  follows easily from
Lemma\,\ref{incgam}.
Similarly for $\int_{2m}^\infty e^{-t}\,t^{m-1}f(y\,t)\,dt$\,.\vspace{-1mm}
\end{proof}\vspace{-4mm}
\begin{Rem}	  \label{remintrem}	     % Rem.2.2
Going a little further with similar arguments, one can show a slightly
extended version of Lemma\,\ref{intrem}. Let $\gamma,c'>0$\,. Restricting
to $0<y\le c'$ it is enough to assume that $f(t)\,t^{-\alpha}$ is bounded on
$]0,1]$ and $f(t)\,t^{-\gamma}$ is bounded on $[1,\infty[$\,.
For $y\,m\ge1$ the conclusion of the lemma still holds when $\alpha$ is
replaced by some $\alpha'\ge\alpha$  (the bounding constant may change).
\end{Rem}
\begin{Lem}	 \label{intasym} %Integralabschätzung Asymptotik  L.2.3
Let $\alpha\ge0\,,\,m_0>0$ (real) be fixed and assume that \,$f\!:\,]0,\infty[\to\C$
is\linebreak
$4$-times differentiable such that $f^{(k)}(t)\;t^{k-\alpha}$ is bounded on
$]0,\infty[$ for $k=0,\cdots,4$. \vspace{-2mm}Then
\begin{multline*}
\sup\,\biggl\{\ \Bigl\lvert\,\int_0^\infty e^{-t}\,t^{m-1}f(y\,t)\,dt\,-\,
\Gamma(m)\,\Bigl(f(ym)+\frac{y^2m}2\,f''(ym)\Bigr)\,\Bigr\rvert\;
\frac{m^2}{y^\alpha\,\Gamma(m+\alpha)}\;:\\
y>0\,,m\ge m_0\,\biggr\}\qquad
\text{is finite.}
\end{multline*}
\end{Lem}
\begin{proof}
By Lemma\,\ref{intrem}, it will be enough to estimate
$\int_{\frac m2}^{2m} e^{-t}\,t^{m-1}f(y\,t)\,dt$\,. We use Taylor expansion
of $f$ around $t_0=y\,m$\,: \
$f(y\,t)=f(y\,m)+y\,(t-m)\,f'(y\,m)+\linebreak
\cdots+\frac16\,(y\,(t-m))^3f'''(y\,m)+(y\,(t-m))^4\,r(t)$
\,with \;$r(t)=\frac1{24}\,f^{\text{\sc iv}}(\xi)$,
\,where\linebreak
$\frac{y\,m}2<\xi<2\,y\,m$\,, hence $r(t)=O((y\,m)^{\alpha-4})$. The
central moments
$\gamma_k=\linebreak
\dfrac1{\Gamma(m)}\int_0^\infty e^{-t}\,t^{m-1}(t-m)^k\,dt$ \,are
\,$1,0,m,2m,3m(m+2)$ \,for $k=0,\cdots,4$. By Lemma\,\ref{intrem}
(taking $\beta=2+k$), we have
$\int_0^{\frac m2} e^{-t}\,t^{k+m-1}\,dt+
\int_{2m}^\infty\ \cdots\ =\linebreak
O((m+1)^{-2-k}\Gamma(m+k))=O(m^{-2}\Gamma(m))$, hence
$\bigl\lvert\Gamma(m)\,\gamma_k-
\int_{\frac m2}^{2m} e^{-t}\,t^{m-1}(t-m)^k\,dt\bigr\rvert=
O(m^{-2}\Gamma(m))$. Integrating now separately the
summands of the Taylor expansion, the approximation by $\gamma_k$ gives
for $f',f'''\text{ and }r$ terms that are $O(y^\alpha m^{\alpha-2}\Gamma(m))$
and for $k=0,\cdots,4$ the approximation error contributes a remainder of
\vspace{.5mm}this size.
Recall that
\;$\sup\bigl\{\dfrac{m^\alpha\Gamma(m)}{\Gamma(m+\alpha)}:\,
m\ge m_0\bigr\}$
is finite (\cite{O}\,(5.02)\,p.\,119), leading to our claim.\vspace{-1mm}
\end{proof}
%			Beweis Th.1.3
\begin{proof}[\bf Proof of Theorem\,\ref{apcoprinc}\ --\ general case]\ \\
We treat here the cases where Lemma\,\ref{intasym} can be applied directly
to $f(t)=W_{\!n,\,\ell+\frac12}(t)$ with $\alpha=\frac12$\,. Recall the
asymptotic behaviour of the Whittaker functions. We have
$W_{\nu\!,\rho}(t)\sim e^{-\frac t2}\,t^\nu$ for $t\to\infty$
(\cite{DLMF}\,Eq.\,13.14.21), hence for $t\in[1,\infty[$ the boundedness
condition holds for all $\nu,\rho,\alpha$\,.
Furthermore,\vspace{.5mm} for $0<t\le1$,
\,$W_{\nu\!,\rho}(t)=
\Bigl(\dfrac{\Gamma(-2\rho)}{\Gamma(-\nu-\rho+\frac12)}\:t^{\rho+\frac12}+
\dfrac{\Gamma(2\rho)}{\Gamma(-\nu+\rho+\frac12)}\:t^{-\rho+\frac12}\Bigr)\,
(1+O(t))$\vspace{.5mm} when $2\rho\notin\Z$\,, by \cite{VK}\,p.\,221\,(2) and
\cite{DLMF}\,Eq.\,13.14.14. In the notation of Theorem\,1, this gives
$W_{\!n,\,\ell+\frac12}(t)=O(\sqrt t)$ for $t\to0+$\,, whenever
$\lambda\neq0$\, and \,$n,\lambda\in\R$\,. In addition, by
\cite{DLMF}\,Eq.\,13.14.15, $W_{\!n,\,0}(t)=O(\sqrt t)$ for $t\to0+$\, holds
for $n=\frac12\,,\,\frac32\,,\cdots$. For the derivatives, we use
\cite{VK}\,p.\,218\,(3),
$W'_{\nu\!,\rho}(t)=
-\frac1t\,W_{\nu+1\!,\rho}(t)-(\frac\nu t-\frac12)\,W_{\nu\!,\rho}(t)$ and
it follows that in the cases above, $f(t)=W_{\!n,\,\ell+\frac12}(t)$ satisfies
$f^{(k)}(t)=O(t^{\frac12-k})$ for $k=0\,,1\,,\cdots$.
Furthermore, there are the trivial cases
$\ell=-\frac12\,,\,n=-\frac12\,,\,-\frac32\,,\cdots$
\,where, according to our convention
on $\frac1{\Gamma(t)}$\,, everything becomes $0$\, (see also
section\,\ref{disc}\,).
Thus the exceptions
in this part of the proof are $(\ell,n)$ where \,$\ell=-\frac12$
\,and \,$n\notin\frac12+\Z$\,. For the remaining part of the proof, we can
assume that in the case $\ell=-\frac12$ \,we have
\,$n\in\{\frac12\,,\,\frac32\,,\cdots\}$.
\\[1mm plus.5mm]
The formulas simplify slightly when working with $\fr P_{nm}$ instead of
$\fr P_{mn}$\,. By \cite{VK}\linebreak p.\,319\,(6),
$\fr P_{mn}^\ell=(-1)^{m-n}\,\fr P_{nm}^{-\ell-1}$.
Hence it will be enough to show boundedness of the expression
\begin{equation}\label{th1mod}
\biggl(\,\fr P_{nm}^{-\ell-1}(x)\,-\,\dfrac1{m^{\ell+1}\Gamma(n-\ell)}\,
W_{n\!,\,i\lambda}\Bigl(\dfrac{2\,m}x\Bigr)\,\biggr)\;x^{\frac12}\,m^2
\end{equation}
for $x\ge1\,,\,m\in n+\N_0\,,\,m>0$.
\\
By \cite{DLMF}\,Eq.\,13.14.31, $W_{\nu\!,\rho}=W_{\nu\!,-\rho}$\,.
We apply Lemma\,\ref{intasym} to the integral in \eqref{averwhit}, giving
approximations by values of $W_{\!n\!,\,i\lambda}$ and
$W''_{\!n\!,\,i\lambda}$\,.
Recall Whittaker's differential equation (\cite{DLMF}\,Eq.\,13.14.1)
$W''_{\nu\!,\rho}(t)=
(\frac14-\frac\nu t-\frac{\frac14-\rho^2}{t^2})\,W_{\nu\!,\rho}(t)$.
Using this, \eqref{averwhit} gives by easy\vspace{-2mm} computations
\begin{multline}       \label{pappr}
\ \fr P_{nm}^{-\ell-1}(x)=
\Bigl(1-\frac1{x^2}\Bigr)^{\!\frac m2}\Bigl(\frac{x+1}{x-1}\Bigr)^{\!\frac n2}
\frac1{\Gamma(n-\ell)}\,
\frac{\Gamma(m)}{\Gamma(m+\ell+1)}\;\cdot \\
\Bigl(\,
\bigl(1+\frac m{2\,x^2}-\frac nx-\frac{\frac14+\lambda^2}{2\,m}\bigr)\,
W_{\!n\!,\,i\lambda}\bigl(\frac{2\,m}x\bigr)+
O\bigl(\frac{\Gamma(m+\frac12)}{\Gamma(m)\,x^{\frac12}\,m^2}\bigr)\,\Bigr)
\vspace{-8mm}
\end{multline}
uniformly for \;$x\ge1\,,\,m\in n+\N_0\,,\,m>0$\vspace{.5mm}\,.
\\
$\frac{\Gamma(m+\frac12)}{\Gamma(m+\ell+1)}$ is bounded for $m\ge0$\,, by
\cite{O}\,(5.02)\,p.\,119.
$\bigl(1-\frac1{x^2}\bigr)^m\bigl(\frac{x+1}{x-1}\bigr)^n$
is easily seen to be bounded for $x\ge1\,,
m\ge n$ \,($n$ is fixed). This implies that the remainder covered by the
$O$-term will contribute a bounded summand to \eqref{th1mod}.
\\
Similarly, we will dispose of the case where $1\le x\le2$\,. Here one
can estimate separately the two parts of the difference in \eqref{th1mod}.
Observe that the condition $m\in n+\N_0\,,m>0$ implies that $m_0=\inf m>0$\,.
Then the asymptotics of $W_{\nu\!,\rho}$ mentioned above gives
$W_{n\!,\,i\lambda}\bigl(\frac{2\,m}x\bigr)=O((\frac mx)^ne^{-\frac mx})=
O(\frac1{m^3})$
for $1\le x\le2\,,\,m\ge m_0$ \;and boundedness in \eqref{pappr} follows
easily. Similarly for the other part in \eqref{th1mod}.
\\[1mm]
For $x\ge2$\,, we have
\,$\bigl(\frac{x+1}{x-1}\bigr)^{\!\frac n2}=1+\frac nx+O(\frac1{x^2})$ and
$\bigl(1-\frac1{x^2}\bigr)^{\!\frac m2}=1-\frac m{2\,x^2}+O(\frac{m^2}{x^4})$
\,(for $m\le x^2$ the last statement follows from elementary expansions and
for $m\ge x^2$ it holds trivially since the left side is bounded by $1$).
Furthermore, by \cite{O}\,(5.02)\,p.\,119,
$\frac{\Gamma(m)}{\Gamma(m+\ell+1)}=
m^{-\ell-1}(1+\frac{\frac14+\lambda^2}{2\,m}+O(\frac1{m^2}))$  for $m\ge m_0$\,.
Combining these estimates, we get after easy computations\vspace{-1.5mm}
\begin{multline*}
\Bigl(1-\frac1{x^2}\Bigr)^{\!\frac m2}\!
\Bigl(\frac{x+1}{x-1}\Bigr)^{\!\frac n2}\!
\frac{\Gamma(m)}{\Gamma(m+\ell+1)}
\Bigl(1+\frac m{2\,x^2}-\frac nx-\frac{\frac14+\lambda^2}{2\,m}\Bigr)=\\
\frac1{m^{\ell+1}}\Bigl(1+O\bigl(\frac1{m^2}+\frac{m^2}{x^4}+\frac{m^3}{x^6}
\bigr)\Bigr)\qquad\text{uniformly for }m\ge m_0\,,\,x\ge2\,.
%\\[-6.5mm] 
\vspace{-6.5mm}
\end{multline*}
It follows that \eqref{th1mod} reduces to\vspace{-1.5mm}
\[
\frac1{m^{\ell+1}}\;O\bigl(\frac1{m^2}+\frac{m^2}{x^4}+\frac{m^3}{x^6}\bigr)\,
W_{n\!,\,i\lambda}\bigl(\frac{2\,m}x\bigr)\;x^{\frac12}\,m^2\,+\,O(1)\,.
\]
Observe that
$\bigl\lvert m^{-\ell-1}\bigr\rvert=m^{-\frac12}$\,. As explained above, we
have \;$W_{n\!,\,i\lambda}(\frac{2\,m}x)=O((\frac mx)^{\frac12})$
for $m\le x$\,, \ and \;$W_{n\!,\,i\lambda}\bigl(\frac{2\,m}x\bigr)=
O((\frac mx)^{-5})$ \,for
$m\ge x$. From this, uniform boundedness of \eqref{th1mod}
follows easily.
\end{proof}
\begin{proof}[\bf Proof of Theorem\,\ref{apcoprinc}\ --\ exceptions]\
\\
As explained at the beginning of the general case, this concerns now the cases
where \,$\ell=-\frac12$ \,and \,$n\notin\frac12+\Z$\,. They arise for the
representations $T_\chi$\,, where $\chi=(-\frac12,\epsilon)$ with
$\lvert\epsilon\rvert<\frac12$\,. Among these is $T_{(-\frac12,0)}$ which
comes (in the setting of \cite{VK}) from the quasi-regular representation
on the space of right cosets $(ZH)\backslash\widetilde G$\,. By Herz's
``principe de majoration" it plays a special r\^ole in the investigations of
the growth of coefficients \,(see \cite{C}\;Sec.1,
\,$\Xi\bigl((\begin{smallmatrix}e^\tau & 0 \\  0 & e^{-\tau}
\end{smallmatrix}) \bigr)=\fr P_{0\,0}^{-\frac12}(\cosh2\tau)$ is
Harish-Chandra's\linebreak $\Xi$-function for $\SL(2,\R)$\,)
and for $\SL(2,\R)$ this is the only exception in the proof.
In all these cases, we have
\,$\fr P_{mn}^{-\frac12}(x)\,\dfrac{\sqrt x}{\log(x)}\to
\dfrac{\sqrt{2}\cos(\pi m)}\pi$ \,for $x\to\infty$ \,(this holds for all
$m,n$ \,with $m-n\in\Z$\,, \,e.g.\;by using \cite{DLMF}\,Eq.\,15.8.9;
\;for $m\ge n\,,\,m>0$ it
follows also from Theorem\,\ref{apcoprinc}, when it is proved). Copying the
proof for the general case
provides an estimate as in Theorem\,\ref{apcoprinc}, with $m^2$ replaced by
$m^{2-\varepsilon}$ for any fixed $\varepsilon>0$. This is still sufficient
for the applications in the corollaries. We indicate now the argument to get
the estimate as stated in Theorem\,\ref{apcoprinc}.
\\
First the case $1\le x\le m$\,. We can take $f(t)=W_{\!n,\,0}(t)$ which
satisfies $f(t)=O(t^{\frac12}\log(t))=O(t^{\frac13})$ for  $0<t\le1$
(\cite{DLMF}\,Eq.\,13.14.19) and $f(t)=O(t^n e^{-\frac t2})$ for
$t\ge1$\,, similarly for the derivatives (see the first paragraph of the proof
for the general case). Now we can apply the extended
version of Lemma\,\ref{intrem} (see Remark\,\ref{remintrem} -- again
with $y=\frac2x$\,; \;starting with
$\alpha=\frac13$\,, then increasing the exponent in the remainder to
$\frac12$) \;allowing
to reduce in \eqref{averwhit} to
$\int_{\frac m2}^{2m}\ $. In the subsequent argument of the proof of
Lemma\,\ref{intasym} the boundedness condition for $f^{(k)}$ is needed
only in the interval $[\frac mx,\frac{4m}x]\subseteq[1,\infty[$\,, hence we
arrive again at \eqref{pappr}. The remaining part can be done as for the
general case.
\\
For $x\ge m$\,, we take $f_1(t)\!=\!\log(t)\sqrt t$ and
$f(t)=W_{\!n,\,0}(t)+\frac1{\Gamma(\frac12-n)}f_1(t)\ \ (t\!\ge\!0)$.
\linebreak By
\cite{DLMF}\,Eq.\,13.14.19, $f(t)=O(t^{\frac12})$ for $t\le4$ (similarly
for the derivatives) and clearly
$f(t)=O(t)$ for $t\ge1$\,. We can apply the extended
version of Lemma\,\ref{intrem} to $f$ (Remark\,\ref{remintrem},
with $\alpha=\frac12$) and then
Lemma\,\ref{intasym} (now $[\frac mx,\frac{4m}x]\subseteq]0,4]$). We have
$I_1:=\int_0^\infty e^{-t}\,t^{m-1}\,f_1(y\,t)\,dt=
y^{\frac12}\bigl(\Gamma(m+\frac12)\log(y)+\Gamma'(m+\frac12)\bigr)$.
By \cite{DLMF} Eq.\,5.11.2,
\;$\frac{\Gamma'(m+\frac12)}{\Gamma(m+\frac12)}=
\log(m+\frac12)-\frac1{2m+1}+O(\frac1{m^2})=\log(m)+O(\frac1{m^2})$, \,hence
\linebreak$I_1=\Gamma(m+\frac12)\,y^{\frac12}\,(\log(y\,m)+O(\frac1{m^2}))$.
Then \;$I_2:=\frac1{\Gamma(m+\frac12)}
\int_0^\infty e^{-t}\,t^{m-1}\,W_{\!n,\,0}(y\,t)\,dt\linebreak
=-\frac1{\Gamma(\frac12-n)}\,y^{\frac12}\,\log(y\,m)+
\frac{\Gamma(m)}{\Gamma(m+\frac12)}\bigl(f(y\,m)+
\frac{(y\,m)^2}{2\,m}f''(y\,m)\bigr)+O(\frac{y^{\frac12}}{m^2})$. Using
again Whittaker's differential equation and
$\frac{\Gamma(m)}{\Gamma(m+\frac12)}=
m^{-\frac12}(1+\frac1{8m}+O(\frac1{m^2}))$  \,(\cite{O}\,(5.02)
p.\,119), one gets \
$I_2=m^{-\frac12}\,W_{\!n,\,0}(y\,m)\,(1+\frac{y^2\,m}8-\frac{y\,n}2)+
O(\frac{y^{\frac12}}{m^2})$.
Estimating the remaining terms in \eqref{averwhit} as in the general case,
the result follows.
\end{proof}
\begin{proof}[\bf Proof of Corollary\,\ref{apl2}]\     %Bew. Cor.1.4
\\
For $n>1$, there are finitely many $m>0$ that are not covered by the theorem.
But since \,$\fr P_{mn}^\ell(t)\to0$ and \,$W_{\!n,\,i\lambda}(\frac1t)\to0$
for $t\to\infty$\,, this part certainly tends to $0$.
The $l^2$-norm of the remaining part can be estimated easily from the
theorem. The second limit relation follows by estimating the
derivative $W_{\!n,\,i\lambda}'$ as in the proof of the theorem.
\end{proof}
\begin{proof}[\bf Proof of Corollary\,\ref{apfour}]\  %Bew. Cor.1.5
\\
Since $h$ has compact support, it is easy to see that $\{h_x:x\ge1\}$
is bounded in $L^2(\R)$. Hence it will be enough to compute the limit for
$g$ in a dense subset, and then by linearity we can reduce to the case
\,$g=e_{n-\epsilon}^\chi$ for some $n\in\epsilon+\Z$\,. Recall that by
\eqref{defP} \,$t_{m-\epsilon\,n-\epsilon}^\chi
\Bigl(\begin{pmatrix}x & 0 \\  0 & \frac1x \end{pmatrix} \Bigr)=
\fr P_{mn}^\ell(\frac{x^2+\frac1{x^2}}2)$ \,for $x\ge1,\;m,n\in\epsilon+\Z$\,.
Using again bounded\-ness of $\{h_x\}$,\vspace{.2mm plus.1mm} we get from
Corollary\,\ref{apl2} and \eqref{fourbasis}, \
$\lim\limits_{x\to\infty}
\:\bigl(\,T_\chi \Bigl(\begin{pmatrix}x & 0 \\  0 & \frac1x
\end{pmatrix} \Bigr)e_{n-\epsilon}^\chi\mid h_x\,\bigr)\linebreak=\
\lim\:\dfrac{\sqrt2}{x^{-2\ell}}\,
{\displaystyle\sum\limits_{m\in\epsilon+\Z}}
\dfrac{(-1)^{m-n}}{\lvert m\rvert^{\ell+1}\Gamma(\sgn(m)n-\ell)}\;
W_{\sgn(m)n,\,\ell+\frac12}\Bigl(\dfrac{4\,\lvert m\rvert}{x^2}\Bigr)\,
e^{-i\pi m}\,\overline{h\Bigl(\dfrac{2\,m}{x^2}\Bigr)}\ =\linebreak
\lim\,\dfrac2{x^2}\,{\displaystyle\sum\limits_{m\in\epsilon+\Z}}
\widehat{e_{n-\epsilon}^\chi}\Bigl(\dfrac{2\,m}{x^2}\Bigr)\,
\overline{h\Bigl(\dfrac{2\,m}{x^2}\Bigr)}\ =\
(\widehat{e_{n-\epsilon}^\chi}\mid h)$\quad(see also the next proof;
note that, as above, for
$\epsilon=0$ the term coming from $\fr P_{0n}^\ell$ does not contribute
to the limit).
\end{proof}
\begin{proof}[\bf Proof of Corollary\,\ref{apL2}]\   %Bew. Cor.1.6
\\
For any $y>0\,,\ J_y\!:(\alpha_m)\mapsto
\frac 1{\sqrt y}\sum_{m\in\epsilon+\Z}\alpha_m\:c_{[ym,y(m+1)[}$
defines an isometric embedding $J_y\!:l^2(\epsilon+\Z)\to L^2(\R)$.
First we treat the part $[0,\infty[$ of $\R$\,. Recall that
$\re(\ell)=-\frac12$\,.
Hence by Corollary\,\ref{apl2} and \eqref{fourbasis}, it is enough to show
that
\[\sum_{\substack{m\in\epsilon+\Z\\m>0}}
\widehat{e_{n-\epsilon}^\chi}(\frac mx)\:
c_{[\frac mx,\frac{m+1}x[}\ \to\ \widehat{e_{n-\epsilon}^\chi}\qquad
\text{for \ }x\to\infty
\]
in $L^2([0,\infty[\,)$. But it follows easily from \eqref{fourbasis} and
the properties of Whittaker functions mentioned before that
$\widehat{e_{n-\epsilon}^\chi}(y)$ is continuous for $y\neq0$, 
\;$\widehat{e_{n-\epsilon}^\chi}(y)=O(\log(\lvert y\rvert))$ for $y\to0$ and
$\widehat{e_{n-\epsilon}^\chi}(y)=O(\frac1{y^2})$ for
$\lvert y\rvert\to\infty$. Similarly for the opposite
half, using that $\fr P_{mn}^\ell=\fr P_{-m\,-n}^\ell$
\,(\cite{VK}\,6.5.5\,(1)).
\end{proof}\vskip1mm
\section{Discrete series}\label{disc} % Sec.3
\vskip 1mm
\begin{varrem}\label{di31}
Considering the representations $T_\chi$ for general $\ell,\epsilon\in\C$\,,
it turns out (as in \cite{VK}\,6.4.3, see also \cite{HT}\;Prop.\,II.1.2.9)
that for $\ell+\epsilon\in\Z$ \,the closed subspace generated by
the basis vectors $e_n^\chi$ with $n\ge-\ell-\epsilon$ is invariant under
$T_\chi$\,. When $\ell<0\ (\ell\in\R)$ one can define a new inner
product (as in \cite{VK}\,p.\,308\,(8)\,), so that\linebreak
\,$\Bigl(\dfrac{\Gamma(-\ell+\epsilon+n)}
{\Gamma(\ell+\epsilon+n+1)}\Bigr)^{\frac12}\,e_n^\chi$\vspace{1mm} \ for
$n\ge-\ell-\epsilon\,,\,n\in\Z$ \,becomes an orthonormal
basis. It makes the restricted
operators from $T_\chi$ unitary and after completion one gets an
irreducible representation $T_\ell^+$ of $\widetilde G$ resp. $\widetilde{\SU}$
\,(hence $T_\ell^+$ is Naimark equivalent to a subrepresentation of $T_\chi$).
The matrix coefficients, using the new basis, are
\begin{equation} \label{defPd+}
\Cal P_{mn}^\ell=\Bigl(\frac{\Gamma(-\ell+n)\,\Gamma(\ell+m+1)}
{\Gamma(-\ell+m)\,\Gamma(\ell+n+1)}\Bigr)^{\frac12}
\fr P_{mn}^\ell\quad \text{ for }\ m,n\in-\ell+\N_0
\end{equation}
(\cite{VK}\,p.\,321\,(8$'$); again this does not depend on the choice of
$\epsilon\!\!\mod1$\,). We may take $\epsilon=-\ell$\,, then
the new basis is
given by \,$e_n^\ell=\Bigl(\dfrac{\Gamma(-2\ell+n)}
{n!}\Bigr)^{\frac12}\,e_n^{(\ell,-\ell)}$\vspace{.5mm} for $n\in\N_0$ and
\,$\Cal P_{mn}^\ell(\cosh 2\tau)=
(T_{(\ell,-\ell)}(g_\tau)\,e_{n+\ell}^\ell\mid e_{m+\ell}^\ell)=
(T_\ell^+(g_\tau)\,e_{n+\ell}^\ell\mid e_{m+\ell}^\ell)$ for
\,$m,n\!\in-\ell+\N_0\,,\,\tau\ge0$\,. For the case of $\widetilde G$\,,
$g_\tau=\bigl(\begin{smallmatrix}e^\tau & 0 \\  0 & e^{-\tau}
\end{smallmatrix}\bigr)$ and 
by \eqref{defbasis} (real model),
$e_n^{(\ell,-\ell)}(t)=\frac1{\sqrt\pi}(t-i)^n(t+i)^{2\ell-n}$.
\\[.5mm]
Similarly, if $\ell-\epsilon\in\Z$\,, 
\,$e_n^\chi$ with $n\le\ell-\epsilon$ generate a closed subspace invariant
under $T_\chi$\,, and for $\ell<0\ (\ell\in\R)$ there is a new inner
product (see also \cite{VK}\,p.\,321\,(4)\,) giving an irreducible unitary
representation $T_\ell^-$ of $\widetilde G$\,. Coefficients are described by
\begin{equation} \label{defPd-}
\Cal P_{mn}^\ell=\ \Cal P_{-m\,-n}^\ell\ \text{ for }\ m,n\in\ell-\N_0
\end{equation}
(\cite{VK}\,p.\,321\,(8)\,).
$T_\ell^+$ and $T_\ell^-$ where $\ell<0$\,, are called the {\it discrete
series} of representations of $\widetilde G$\,. By arguments 
as in \cite{VK}\,6.4.4, they are mutually non-equivalent.
\end{varrem}
\begin{proof}[\bf Proof of Theorem\,\ref{apcodisc}]\ % Beweis Th.1.8
\\
This is similar to the proof of Theorem\,\ref{apcoprinc}. First, we
approximate $\fr P_{nm}^{-\ell-1}$. By \cite{DLMF}\,Eq.\,13.14.15 \,(recall
that $\ell+n\in\N_0$\,, compare also Remark\,\ref{remdisc}\,(c) below),
$\lvert W_{\!n,\,-\ell-\frac12}(t)\rvert\le c\,t^{-\ell}$ holds for $t>0$\,.
Since $\ell\le-\frac12$\,, we can apply Lemma\,\,\ref{intasym} with
$\alpha=\frac12$\,. Then similar arguments as before give the estimate
\,$\fr P_{nm}^{-\ell-1}(x)=\dfrac1{m^{\ell+1}\,\Gamma(n-\ell)}\,
W_{n,\,-\ell-\frac12}\Bigl(\dfrac{2\,m}x\Bigr)+
O\bigl(\dfrac1{x^{\frac12}\,m^{\ell+\frac52}}\bigr)$\vspace{1mm}
uniformly for \;$x\ge1\,,\,m\in n+\Z\,,\linebreak m\ge n$\,. We use that
$W_{\nu\!,\rho}=W_{\nu\!,-\rho}$ \,(\cite{DLMF}\,Eq.\,13.14.31)\,.
By \cite{O}\,(5.02)\,p.\,119, the additional factor appearing in
\eqref{defPd+} is
\;$\Bigl(\dfrac{\Gamma(n-\ell)}{(n+\ell)!}\Bigr)^{\frac12}m^{\ell+\frac12}
\bigl(1+O(\frac1{m^2})\bigr)$
\;and this produces the desired estimate of $\Cal P_{mn}^\ell$.\vspace{-1.5mm}
\end{proof}
\begin{Rems}	  \label{remdisc}	 % Rem.3.2
(a) \ In the classical case of $\SL(2,\R)$ \,(i.e., $\epsilon=0,\frac12$), the
two conditions $\ell\pm\epsilon\in\Z$ are equivalent.
Note that in \cite{VK} the assignment $\pm$ has been reversed compared to
the earlier edition \cite{V} (which was used in \cite{L1}), hence $T_l$ of
\cite{L1}\,p.\,5 is now $T_{-l-1}^+$ and $e_n^l$ becomes
$e_{n-l-1}^{-l-1}$ (for $l=0,1,\dots,\;n\in l+\N$).
In the notation of \cite{Sa}\,p.\,51, $T_\ell^-$ corresponds to $U^+(\cdot,h)$
with $h=-\ell$ (defined there for $\re h>0$\,, as a representation of the
universal covering group of $\SU(1,1)$\,) and (\cite{Sa}\,p.\,52)
$T_\ell^+$ corresponds to $U^-(\cdot,h)$ for $h=\ell$\,. But on $\SL(2,\R)$,
(due to different transformations) $T_\ell^+$ corresponds to $T^+(\cdot,h)$
with $h=-\ell$ (defined in \cite{Sa}\,p.\,23 for $h=\frac12,1,\dots$)
and $T_\ell^-$ corresponds to $T^-(\cdot,h)$ with $h=\ell=-\frac12,-1,\dots$.
In \cite{VK}\;p.\,458 a different transformation is used for the discrete
series, compared to their general treatment before. Our $T_\ell^+$ corresponds
to their $\hat T_\ell^-$.
In \cite{P}\,p.\,102,
our $T_{-\ell}^+$ is denoted as $D_\ell^+$ and $T_{-\ell}^-$ is $D_\ell^-$
($\ell>0$).
See Remark\;\ref{remcomp}\,(c) for the (non-unitary) generalizations
with arbitrary $\ell\in\C$\,.
\\
Many authors (e.g.\;\cite{Wa}\;p.\,350) use the term discrete series to denote
the irreducible
representations that are square integrable. Note that for $-\frac12\le\ell<0$
the representations $T_\ell^\pm$ are not square integrable\!$\mod Z$ \,and
for $-\frac12<\ell<0$ the representations $T_\ell^\pm$ do not belong
to the reduced dual of $\widetilde G$\,. We follow here the terminology
of \cite{P} and \cite{Sa} which includes what is sometimes called mock
discrete series and limit of discrete series. From the explicit expression
by Jacobi polynomials given in (c), it follows easily that
\,$\fr P_{mn}^\ell(x)\,(\frac x2)^{-\ell}\to\binom{2\ell}{\ell+m}$ for
$x\to\infty$ \;($x\in\R\,,\,\ell\in\C\,,\;m,n\in-\ell+\N_0$\,; similarly
for $m,n\in\ell-\N_0$).
\\[.5mm]
For $\ell=-\frac12$ \,the inner product considered in \ref{di31} is just
the standard one, $e_n^{-\frac12}=e_n^{(-\frac12,\frac12)}$ and
$\Cal P_{mn}^{-\frac12}=\fr P_{mn}^{-\frac12}$\,. Hence (on $\R$) this gives
the
decomposition \,$T_{(-\frac12,\frac12)}=T_{-\frac12}^-\oplus T_{-\frac12}^+$
into irreducible representations.
\\[.5mm]
The representation spaces for $T_\ell^\pm$ ($\ell<0$) can be realized as
spaces of analytic functions. For the realization of $T_{(\ell,-\ell)}^\T$
on $L^2(\T)$ mentioned in Remark\,\ref{remrep}\,(a),\linebreak
$T_\ell^+$ comes from the
subspace generated by $e_n\,,\;n\ge0$\,
giving functions that are analytic outside the closed unit disc
(as in \cite{VK},
but note that the formula for the norm given in \cite{VK}\,p.\,310\,(14$'$)
is not correct; equivalently, \cite{Sa}\;p.\,18ff. uses a space $H_{2,h}(M)$
of conjugate analytic functions on the unit disc and a slightly different
scaling for the inner product). Passage to $\SL(2,\R)$ and $\widetilde G$
(using the transformation
$E_\chi$ as in Remark\,\ref{remrep}\,(b)\,) leads to spaces $\Cal H_\ell^+$
of analytic functions
in the upper half plane. For $\ell<-\frac12\,,\ \Cal H_\ell^+$ is
a weighted Bergman space, \,the norm  described in \ref{di31}
is given by
\;$\lVert f\rVert^2_{\ell+}=\frac4{\Gamma(-2\ell-1)}\,
\int_{y>0}(4y)^{-2\ell-2}\lvert f(x+iy)\rvert^2\,d(x,y)$\,.
For $-\frac12\le \ell<0$ a different description of the norm is needed,
e.g. as a reproducing kernel Hilbert space (see \cite{Sa}\,p.\,19). For
$\ell=-\frac12$ one gets the classical $H^2$-space of the upper half plane.
Applying the Fourier transform (and its extensions, see the comments
after Corollary\,\ref{apl2}) one gets for these functions $\hat f(t)=0$
for $t<0$ \vspace{.5mm}and the norm defines a weighted
$L^2$-space on $[0,\infty[$\,:
\;$\lVert f\rVert^2_{\ell+}=
\int_0^\infty \lvert \hat f(t)\rvert^2(\frac t2)^{2\ell+1}dt$ \;(this works
for all $\ell\!<\!0$). In particular for the basis $(e_n^\ell)$ defined in
\ref{di31}, using \eqref{fourbasis} and the expression
for $W_{\!n,\,\ell+\frac12}$ given in (c) below, results in
\;$\widehat{e_n^\ell}(y)\:=\ e^{i\pi\ell}\,
(\frac{2\,n!}{\Gamma(-2\ell+n)})^\frac12\,e^{-y}\,y^{-2\ell-1}
\,L_n^{(-2\ell-1)}(2 y)$ \,for $y>0\,,n\in\N_0$ \,(as mentioned after
\eqref{fourbasis} a similar description holds for
$(e_n^{(\ell,-\ell)})\sphat$ when $\re(\ell)<0$). For $\ell<0$ the
orthonormality
relations are equivalent to the classical properties of Laguerre polynomials
(\cite{VK}\,p.\,256\,(1)\,).
\\[.5mm plus .5mm]
Similarly, for $T_\ell^-,\ \ell<0$ one gets spaces $\Cal H_\ell^-$ of analytic
functions in the lower half plane
and after Fourier transform, functions supported on $]-\infty,0]$\,.
For $\epsilon=\ell$ a basis is
given by \;$e_{n-}^\ell\!=\Bigl(\dfrac{\Gamma(-2\ell-n)}
{(-n)!}\Bigr)^{\frac12}\,e_n^{(\ell,\ell)}$ for $n\in-\N_0$ satisfying
$\Cal P_{mn}^\ell(\cosh 2\tau)=
\bigl(T_{(\ell,\ell)}(g_\tau)\,e_{(n-\ell)-}^\ell\vert
e_{(m-\ell)-}^\ell\bigr)_{\ell-}$ for
$m,n\in\ell-\N_0\,,\,\tau\ge0$\,. \vspace{.5mm}For $\widetilde G$\,,
\,$g_\tau=\bigl(\begin{smallmatrix}e^\tau & 0 \\ \! 0 & e^{-\tau}
\end{smallmatrix}\bigr)$ and in the real model
$e_{n-}^\ell=\overline{e_{-n}^\ell}$ on~$\R$\,. For the analytic \vspace{.5mm}
extension one gets the reflected function
$e_{n-}^\ell(z)=\overline{e_{-n}^\ell(\bar z)}$ for
$\re(z)\le0\,,\,n\in-\N_0$\,.
\;$T_\ell^-$ is unitarily equivalent to the conjugate (contragradient)
representation of $T_\ell^+$ (compare \cite{VK}\;p.\,306).
As mentioned in Remark\;\ref{remrep} and in \ref{di31}, in the
real model $T_\ell^\pm$ and the basis above do not depend on the choice of
$\epsilon$ (in the coset$\mod1$). Modifying $\epsilon$ changes just the
index set for the basis. In the model \vspace{.5mm}on $\T\,,\ 
T_{(-\frac12,\frac12)}^{\T}\neq T_{(-\frac12,-\frac12)}^{\T}$\,, hence the
decomposition of $T_{(-\frac12,\frac12)}^{\T}$ is just unitarily equivalent
to $T_{-\frac12}^-\oplus T_{-\frac12}^+$ (resp. their models on $\T$).
\\[.5mm]
$L_k^{(\pm\frac12)}$ are closely related to the Hermite polynomials
(\cite{DLMF}\,Eq.\,18.7.19\,,\;Eq.\linebreak18.7.20) which gives other
models of $T_{-\frac14}^+\,,T_{-\frac34}^+$ described in \cite{HT}\;III.2.1.
\vskip1.5mm plus 1.6mm
\item[(b)] \ By similar arguments as in the proof of Theorem\,\ref{apcodisc}
(taking $\alpha=-\ell$) one can show a corresponding estimate, \,with
$x^{\frac12}\,m^2$ replaced by
$x^{-\ell}\,m^{\ell+\frac52}$, working for any (fixed) $\ell<0$\,.
Corollary\,\ref{apl2} extends without problems
\pagebreak(adjusting the expression according to Theorem\,\ref{apcodisc}).
\\[.3mm]
For an analogue of Corollary\,\ref{apfour}, we take a continuous
function $h$ with compact support in $]0,\infty[$ and put \
$h_x=\frac{\sqrt2 e^{-i\pi\ell}}{x^{2\ell+2}}\,
\sum_{m\in\N_0}(-1)^m\,h(\frac{2m}{x^2})\,m^{\ell+\frac12}\,e_m^\ell$\,
\vspace{.5mm}
(where\linebreak $\ell<0\,,\,x\ge1$). Then for $x\to\infty$, it follows as
\vspace{.5mm}before that \
$\bigl(T_\ell^+\bigl((\begin{smallmatrix}x & 0 \\  0 & \frac1x
\end{smallmatrix}) \bigr)g\mid h_x\bigr)\ \to\
\int_0^\infty\hat g(t)\,\overline{h(t)}\,(\frac t2)^{2\ell+1}\,dt$
\,(exhibiting the inner product for $T_\ell^+$ on the Fourier side
as described in (a)\,). Similarly in the case of Corollary\,\ref{apL2}.
\vspace{.8mm}For $x\to\infty$\,, one has
$\sqrt x\,\sum_{m\ge-\ell}(-1)^{m-n}\,\Cal P_{mn}^\ell(x)\:
c_{[\frac mx,\frac{m+1}x[}(t)\ \to\
e^{i\pi n}\,(\frac t2)^{\ell+\frac12}\,\widehat{e_{n+\ell}^\ell}(t)$
\,in $L^2([0,\infty[)$.
\vskip2mm plus 1mm
\item[(c)] \ One has
\,$\fr P_{mn}^\ell(x)=2^m(x-1)^{\!\frac{n-m}2}(x+1)^{\!-\frac{m+n}2}
P_{\ell+m}^{(n-m,-m-n)}(x)$ \,for $x>1\,,\linebreak\ell,m,n\in\C$ with
$\ell+m,n-m\in\N_0$ \,(as in \cite{VK}\,p.\,320\,(2)),
where $P_k^{(\alpha,\beta)}(x)$  denotes the
{\it Jacobi polynomials}. Furthermore, by
\cite{DLMF}\,Eq.\,13.14.3\,,\,Eq.\,13.14.31 and
\cite{O}\,p.\,259,\,Ex.\,10.3,
\,$W_{\!n,\,\ell+\frac12}(x)=
e^{-\frac x2}x^{-\ell}(\ell+n)!\,(-1)^{\ell+n}L_{\ell+n}^{(-2\ell-1)}(x)$ for
$x>0\,,\linebreak \ell,n\in\C$ with $\ell+n\in\N_0$\,, where
$L_k^{(\alpha)}(x)$
denotes the {\it Laguerre polynomials} \,(\cite{DLMF}\- Eq.\,18.5.12).
Thus the first part of the proof of Theorem\,\ref{apcodisc} produces an
approximation on $[1,\infty[$ of Jacobi polynomials in terms of
Laguerre polynomials. With the transformation $y=\dfrac{x-3}{x+1}$
\,(mapping $[1,\infty[$ to $[-1,1[$\,) and using \cite{VK}\,p.\,291\,(8),
one gets an approximation for $y\in[-1,1]$. This gives a refinement of
\cite{VK}\,p.\,298\,(4). In \cite{Du} Dunster constructs approximations
of Jacobi polynomials using a
method of Olver. \cite{Du}\,(2.43) amounts
(with some rewriting) to approximating  $P_h^{(p,q)}(y)$\vspace{.8mm}
\;by\newline
$\dfrac{2^{p+\frac{q+1}2}(h+q+\frac{p+1}2)^{-\frac p2}}
{\sqrt{h+\frac{p+q+1}2}}\;
\dfrac{e^{-\xi_1}\,\xi_1^{\frac{p+1}2}}
{(1-y)^{\frac p2+\frac14}(1+y)^{\frac q2}}\;
\bigl(\dfrac{2h+p+1-\xi_1}{(y-x_t)\,\xi_1}\bigr)^\frac14\:L_h^{(p)}(2\xi_1)$
\ \vspace{.8mm}where\linebreak
$h\in\N_0\,,\;p,q>0\,,\;
x_t=2\,(\frac q{2h+p+q+1})^2-1\,,\
\xi_1=(h+\frac{p+q+1}2)\,\xi$ and \,$\xi\;(=\xi(y))$ \vspace{.5mm}is defined
by a Liouville
transformation (\cite{Du}\,(2.10)). This gives a better approximation, but
$\xi$ is not very explicit. Basically, our approximation (in the case
of $y\in[-1,1]$) amounts to replace $\xi_1$ by a rational linear approximation
$\xi_{app}=(h+q+\frac{p+1}2)\frac{1-y}{3+y}$ which gets quite close outside
some neighbourhood of $1$\,.
\\[.3mm plus.4mm]
Similarly as above (\cite{VK}\,p.\,322\,(1)), $\fr P_{mn}^\ell$ can be
expressed by {\it Jacobi functions} for arbitrary $\ell,m,n\in\C$ with
$m-n\in\Z$
(but be aware that \cite{VK} and \cite{Du} use different scalings for Jacobi
functions).
\cite{Du} gives similar approximations for Jacobi functions
$P_h^{(p,q)}$ by Whittaker functions, restricting to $h,p,q\ge0$ (real).
In fact, the method
works more generally, giving also slightly improved approximations (compared
to our method) in the
case of the principal series, complementary series and the uniformly bounded
representations.
\\[1.5mm plus .5mm]
Various authors gave approximations (and series expansions) using Bessel
functions. From Jones \cite{J}\;(59)\;p.\,384 one gets asymptotics for
$\fr P_{nn}^{-1/2+i\lambda}$ when $n\to\infty$\,.\linebreak
Another approach
is \cite{BG}\;Thm.\,1.2, building on work of Stanton, Thomas,\linebreak
Fitouhi, Hamza
(related is \cite{VK}\;p.\,298,\,(1)-(3)\,).
\\
Going back to the principal series, the Harish-Chandra decomposition
(\cite{Kn}\;Th. 8.32) reduces for $\SL(2,\R)$ and  $\lambda\neq0$ to the
transformation of the hypergeometric function used by Bargmann
([Ba]\;(10.23)\,). The corresponding series expansion converges on
$]1,\infty[$\,.
The partial sums can be used to improve for ``very large $x$" the
approximation obtained from the ``constant term" (see Remark\;\;1.7\,(a)).
An error estimate for the partial sums has been given in
\cite{Bk}\;Th.\,12.1. The error increases with $m$\,. When $x$ gets smaller
the behaviour of the partial sums becomes more complicated, more and more
terms are needed to get below some bound, making the method less useful.
\\
Looking at the asymptotics of Whittaker functions (see the beginning of the
proof of Theorem\;\ref{apcoprinc}), one can see that for $\ell,m,n$ fixed
the approximation of Theorem\;\ref{apcoprinc} does not lead to the ``constant
term of the Harish-Chandra expansion" for $x\to\infty$ (rather giving an upper
bound for the decay). Similarly, for $x$ fixed, the speed of decay to $0$
for $m\to\infty$ is not detected by the Whittaker approximation.
\\
Comparing now the approximations by Bessel, Whittaker and
Harish-Chandra/ Bargmann expansion (partial sums): \ all three are
applicable throughout $]1,\infty[$\,, but their ``useful scopes" are somehow
complementary. For $\ell,n$ fixed, $\lvert m\rvert$ large, Bessel serves
best (roughly) for  $x$ small compared to $m$\,, Harish-Chandra is most
useful for  $x$ large compared to $m$\,, and Whittaker for $x$ of 
``comparable size" to $m$\,.
\\
Similarly for the discrete series. The analogue of the ``constant term" is
to take just the term of maximal exponent from the Jacobi polynomial
(determining the asymptotics for $x\to\infty$) and then the methods are
performing as above.
\end{Rems}\vskip1mm
\section{Complementary series, uniformly bounded representations}\label{compl} % Sec.4
\vskip 1mm		  % Sec.4
\begin{varrem}\label{di41}
Similar constructions as in \ref{di31} can be done when
$\ell,\epsilon\in\R\,,\ -1<\ell<0$ and
$\lvert\epsilon\rvert<\frac12-\lvert\frac12+\ell\rvert$.
One can define a new inner product, so that
\,$e_n^{\chi,c}=\Bigl(\dfrac{\Gamma(-\ell+\epsilon+n)}
{\Gamma(\ell+\epsilon+n+1)}\Bigr)^{\frac12}e_n^\chi %\linebreak
\ \;(\,=\bigl(\frac{\,\Gamma(-\ell-\epsilon-n)\; \sin(\pi(\ell+\epsilon))}
{\Gamma(\ell-\epsilon-n+1)\sin(\pi(\ell-\epsilon))}
\bigr)^{\frac12}e_n^\chi\,)$ for
$n\in\Z$ \,becomes an\vspace{.5mm} orthonormal
basis. It makes the
operators defined by $T_\chi$ unitary and after completion, one gets an
irreducible unitary representation $T_\chi^c$ of $\widetilde G$
\,(as in \cite{VK}\,(6.4.2)\,;
hence $T_\chi^c$ is Naimark equivalent to $T_\chi$).
This is called the {\it complementary
series} of representations of $\widetilde G$\,.
$T_{(\ell,\epsilon)}^c$
is equivalent to $T_{(-\ell-1,\epsilon)}^c$ (see also \cite{Sa}\,p.\,44).
Apart of this, the representations $T_\chi^c$ are mutually non-equivalent
(as in \cite{VK}\,(6.4.4)\,)
for $\ell,\epsilon$ as above. For $\ell=-\frac12$ the inner product
is unchanged and this intersects with the principal series
(section\,\ref{Main}). Together with the trivial representation, this
gives all the irreducible unitary representations of $\widetilde G$
\,(up to equivalence). \cite{P} denotes $T_\chi^c$ as $E_q^{(\tau)}$ with
$q=-\ell(\ell+1)\,,\;\tau=\epsilon \text{ or }\epsilon+1\,,\;0\le\tau<1\,,\;
\tau(1-\tau)<q\le\frac14$\,.
In the classical case of $\SL(2,\R)$, the conditions for the complementary
series given above enforce $\epsilon=0$\,.
\\
The Hilbert space $\Cal H_{\chi,c}$ for $T_\chi^c$ (with norm
$\lVert\cdot\rVert_{\chi,c}$) can be realized (working on $\R$) in a similar
way as in Remark\,\ref{remdisc}\,(a). For $-1<\ell<-\frac12$ \,(hence
$\lvert\epsilon\rvert<1+\ell$), using the transformation $E_\chi$ as before,
one gets for a \vspace{.5mm}bounded $L^2$-function $f\!:\R\to\C$\,,
\;$\lVert f\rVert^2_{\chi,c}=
\iint_{\R^2}k_\chi(x-y)f(y)\,\overline{f(x)}\,d(x,y)$ \vspace{.5mm}with
$k_\chi(x)=\frac{e^{\scriptstyle i\mspace{1mu}\epsilon\pi\sgn(x)}}
{\Gamma(-2\ell-1)\sin(\pi(-\ell-\epsilon))}\;
\lvert 2x\rvert^{-2\ell-2}$ \,(up to a constant this is the inner product
of \cite{Sa}\,(2.4.9) defining the Hilbert space $\Cal H'_{\sigma,h}$
with $\sigma=-\ell-\frac12\,,\,h=\epsilon$). Similarly for $-\frac12<\ell<0$
if one adds the assumption $f'\in L^2$ (distributional derivative; also
using the distributional version of $\lvert x\rvert^{-2\ell-2}$,
\cite{VK}\;3.1.6).
\\
Combined with Fourier transform one gets
a weighted $L^2$-space on $\R$\;: \ $\lVert f\rVert^2_{\chi,c} =
\int_{-\infty}^\infty \lvert \hat f(t)\rvert^2\,\rvert\frac t2\rvert^{2\ell+1}
\,\frac{\sin(\pi(\ell+\sgn(t)\epsilon))}{\sin(\pi(\ell+\epsilon))}\:dt$
\,(\vspace{.5mm}this works for all $-1<\ell<0\,,\,
\lvert\epsilon\rvert<\frac12-\lvert\frac12+\ell\rvert$\,; compare
\cite{KS}\;(7.5)\,). By \cite{KS}\;L.\,20, for
$-1<\ell\le-\frac12$ we have \,$L^p(\R)\subseteq\Cal H_{\chi,c}$ with
$p=-\frac1\ell$ (bounded embedding) and by duality one gets for
$-\frac12\le\ell<0$ that $\Cal H_{\chi,c}\subseteq L^p(\R)$\,. 
\\
For $\ell=\ell_\R+i\lambda$ put $\chi_\R=(\ell_\R,0)$. For $-1<\ell_\R<0\,,\;
\epsilon,\lambda\in\R$\,, one can show
similarly as in \cite{ACD}\;L.\,4.3
(using the Mellin transform) that the operators
$T_\chi(g)\ (g\in\widetilde G)$ are {\it uniformly bounded} for
$\lVert\cdot\rVert_{\chi_\R,c}$ \,(in principle this can also be deduced
from the formulas in \cite{Sa}\;Ch.\,3,\,\S1,2). As in \cite{KS}\;L.\,18,
one can see that for all $\chi$ with $-1\le\re(\ell)\le0\,,\;\epsilon\in\R$
the operators $T_\chi(g)\ (g\in\widetilde G)$ are isometric for
$\lVert\cdot\rVert_p$ with $p=-\frac1{\re(\ell)}$\,, hence $T_\chi$ extends
also to a representation on $L^p(\R)$\,.
\\[2mm plus.5mm]
For $\re(\ell)>-\frac12\,,\ \ell\notin\frac12\Z\,,\ m-n\in\Z$\,, one has
\,$\fr P_{mn}^\ell(x)\,(\frac x2)^{-\ell}\to
\frac{\Gamma(2\ell+1)}{\Gamma(\ell+m+1)\Gamma(\ell-m+1)}$ for $x\to\infty$
\;(with the convention $\frac1\infty=0$ as before). For $\re(\ell)<-\frac12$
one can use again that $\fr P_{mn}^\ell=(-1)^{m-n}\,\fr P_{nm}^{-\ell-1}$\,.
Compare also Remark\,\ref{remdisc}\,(a)\, for the cases \,$m,n\in-\ell+\N_0$
or \,$m,n\in\ell-\N_0$ with $\ell\in\C$ (discrete series)
and Remark\,\ref{remprinc}\,(a)\, for $\re(\ell)=-\frac12$\,.\vspace{-.5mm}
\end{varrem}
\begin{varrem}\label{di42}	 % 4.2
Analogously to section\,\ref{disc},  $\Cal P_{mn}^\ell$ is defined as in
\eqref{defPd+}, but now with \,$m,n\in\epsilon+\Z$
\,(\eqref{defPd-} holds as well for such $m,n$).
Similarly as in Theorem\,\ref{apcodisc} (see also
Remark\,\ref{remprinc}\,(b)\,) one can show that for fixed $\ell,n\in\R$
with $-1<\ell<0\,,\;n\in\epsilon+\Z$ and
$\lvert\epsilon\rvert<\frac12-\lvert\frac12+\ell\rvert$\,,
one has \;$\Cal P_{mn}^\ell(x)\,=\,\frac{(-1)^{m-n}\sgn(\Gamma(n-\ell))}
{{\textstyle(}m\,\Gamma(n-\ell)\,\Gamma(n+\ell+1){\textstyle)}^{\frac12}}\,
W_{n\!,\,\ell+\frac12}\bigl(\frac{2\,m}x\bigr)
+\,O\bigl(\frac1{m^{\frac52}}
(\frac xm)^{\lvert\ell+\frac12\rvert-\frac12}\bigr)$ \vspace{.5mm}\,uniformly
for all
$m\in n+\N_0$ with $m>0$ and all $x\ge1$ \,(the condition on $\epsilon$
implies \,$\sgn(\Gamma(n-\ell))=
\sgn(\Gamma(n+\ell+1))$\,).
It follows (as an analogue to Corollary\,\ref{apL2}, see also
Remark\,\ref{remdisc}\,(b)\,) that
for $x\to\infty$\,, one has
$\sqrt x\,\sum_{m\in\epsilon+\Z}(-1)^{m-n}\,\Cal P_{mn}^\ell(x)\:
c_{[\frac mx,\frac{m+1}x[}(t)\ \to\
e^{i\pi n}\,(\frac t2)^{\ell+\frac12}\bigl(
\frac{\sin(\pi(\ell+\sgn(t)\epsilon))}{\sin(\pi(\ell+\epsilon))}
\bigr)^{\frac12}
\vspace{.5mm}\,\widehat{e_{n-\epsilon}^{\chi,c}}(t)$
\,in $L^2(\R)$. Similarly there are analogues of the
Corollaries\;\ref{apl2},\,\ref{apfour}.
\\[2mm]
For $\ell\in\C\,,\;\epsilon\in\R$ consider $\ell_\R\,,\,\chi_\R$ as in
\ref{di41}. For $\ell\notin\R$ the $e_n^\chi$ are no longer orthogonal
in $\Cal H_{\chi_\R,c}$\,. But by uniform boundedness of $T_{\chi}$
(for $\lVert\cdot\rVert_{\chi_\R,c}$) there exists an equivalent
Hilbert norm on $\Cal H_{\chi_\R,c}$ making $T_{\chi}$ unitary on
$\widetilde K$ and it follows that $(e_n^\chi)_{n\in\Z}$ is still a basis of 
$\Cal H_{\chi_\R,c}$ in the Banach space sense and
$\bigl(\fr P_{m+\epsilon\,n+\epsilon}^\ell(\cosh2\tau)\bigr)_{m,n\in\Z}$ gives
the corresponding matrix representation for (the extension of) the operators
$T_\chi(g_\tau)$ on $\Cal H_{\chi_\R,c}$\,, where (for $\widetilde G$)
\,$g_\tau=\bigl(\begin{smallmatrix}e^\tau & 0 \\ \!0 & e^{-\tau}
\end{smallmatrix}\bigr)$. Using again Remark\,\ref{remprinc}\,(b), one can
also prove asymptotic properties for the columns of these matrices.
\vspace{-3mm plus.5mm}
\end{varrem}
\begin{Rems}	  \label{remcomp}	  % Rem.4.3
(a) \ For $\alpha<\frac12$  we consider (as in \cite{ACD}\;p.\,135) the space
$H^\alpha=\{f: \lvert t\rvert^\alpha\hat f(t)\in L^2(\R)\}$ \,a
(Hilbert-)\,subspace
of $\Cal S'(\R)$ (tempered distributions), related to a fractional Sobolev
space (in \cite{KS}\;L.\,20,\,p.\,32 this is $H_\sigma$ with
$\sigma=\alpha-\frac12\,,\linebreak 0<\sigma<\frac12$). For $a,b>0$ there are
additional norms (and corresponding inner products)
$\lVert f\rVert_{H^\alpha_{a,b}}^2=
a\,\int_0^\infty\bigl\lvert\, t^\alpha\hat f(t)\bigr\rvert^2dt+
b\,\int_{-\infty}^0
\bigl\lvert\,\lvert t\rvert^\alpha\hat f(t)\bigr\rvert^2dt$\,. Then, as
mentioned in \ref{di42}, for $\ell,\epsilon\in\R\,,\ -1<\ell<0$ and
$\lvert\epsilon\rvert<\frac12-\lvert\frac12+\ell\rvert$, one has
$\Cal H_{\chi,c}=H^{\ell+\frac12}_{a,b}$ with
$(a,b)=
2^{-2\ell-1}(1,\frac{\sin(\pi(\ell-\epsilon))}{\sin(\pi(\ell+\epsilon))})$.
For each $\ell$ the range of $\epsilon$ covers all $b>0$\,.
\\[2.5mm plus.5mm]
Consider now $\ell\in\C$ with $\ell_\R=\re(\ell)$ satisfying $-1<\ell_\R<0$\,.
Then for each $\epsilon\in\R$ and $a,b>0$\,, taking $\chi=(\ell,\epsilon)$,
$T_\chi$ defines a {\it uniformly bounded representation} $T_\chi^{a,b}$ of
$\widetilde G$ on $H^{\ell_\R+\frac12}_{a,b}$
\,(for $\SL(2,\R)$, i.e. $\epsilon=0,\frac12$, \cite{KS}\;p.\,20, taking up
a construction of \cite{EM}, introduced
this for $a=b=1$\,, using the notation $v^\pm(\cdot,s)$ with $s=\ell+1$;
then in \cite{ACD}\;p.\,136 this
is denoted for general $a,b>0$ as $\pi_{\lambda,\epsilon'\!,a,b}$ where
$\lambda=\ell+\frac12\,,\epsilon'=2\epsilon$). For $(a',b')=\alpha_0(a,b)$
with $\alpha_0>0$ it is easy to see that the representations are unitarily
equivalent. Otherwise (for $\chi$ fixed, excluding $(-\frac12,\epsilon)$ with
$\epsilon\in\frac12+\Z$) they are just boundedly equivalent
(\cite{ACD}\;p.\,128 calls it similar). Note that the range of $\epsilon$
occuring for these uniformly bounded representations exceeds that arising
for the complementary series. In particular, for
$\ell\in\R\,,\;-1<\ell<0\,,\;\epsilon\in-\ell+\Z\,,\linebreak
\;T_\chi^{a,b}$ contains
$T_\ell^+$ from the discrete series as a proper subrepresentation
(for $a=2^{-2\ell-1}$, see Remark\;\ref{remdisc}\,(a); otherwise a unitarily
equivalent copy). Similarly for $\epsilon\in\linebreak\ell+\Z\,,\;T_\chi^{a,b}$ contains
$T_\ell^-$. For
$\ell=-\frac12\,,\;\epsilon\in\frac12+\Z\,,\;T_\chi^{a,b}\cong T_\chi$ for
all $a,b>0$\,.
\\[2.5mm plus.5mm]
For a uniformly bounded representation $T$ of $\widetilde G$ put
$\lVert T\rVert_{ub}=\sup_{g\in\widetilde G}\lVert T(g)\rVert$\,.
In the case of $\SL(2,\R)$\,, i.e.\;$\epsilon=0,\frac12$\,, \;\cite{ACD}
proved
various results on $\lVert T_\chi^{a,b}\rVert_{ub}$\,. For $\chi$ fixed,
\cite{ACD}\;Th.\,1.1 shows that $\lVert T_\chi^{a,b}\rVert_{ub}$ gets
minimal for $a=b$ and for $\chi\neq(-\frac12,\frac12)$ the minimum is strict.
This is no longer true for general $\epsilon$\,. For
$\epsilon\notin\pm\ell+\Z$
(these are the cases where $T_\chi$ is irreducible, see (c) below) the
minimal norm is attained for
$\frac ba=\big\lvert\frac{\sin(\pi(\ell-\epsilon))}
{\sin(\pi(\ell+\epsilon))}\big\rvert$
\,(and this is also strict). For \,$\im\ell\to\infty$ ($\epsilon$ fixed),
this optimal value for $\frac ba$ tends to $1$\,. For $\epsilon\in-\ell+\Z$
the infimum of the norms is $1$\,, but it is not attained. For
$\lVert T_\chi^{a,b}\rVert_{ub}$ approaching the infimum with $b$ bounded,
one has $a\to0$\,, but the ``limit case" $T_\chi^{0,b}$ ($b>0$) falls outside
the range considered above (it corresponds to the quotient representation
induced by $T_\chi$ on $\Cal H_{\chi_\R,c}/\Cal H_\ell^+$ which is
unitarily equivalent to $T_{-\ell-1}^-$). Similarly for
$\epsilon\in\ell+\Z$\,.
For $\ell$ {\it real} ($-1<\ell<0$),
$\lvert\epsilon\rvert<\frac12-\lvert\frac12+\ell\rvert$,
the minimal norm is $1$\,, attained by the complementary series
representation $T_\chi^c$\,. For\vspace{1mm}
$\frac12-\lvert\frac12+\ell\rvert<\lvert\epsilon\rvert\le\frac12$\,,
the optimal value is \;$\min\,\lVert T_\chi^{a,b}\rVert_{ub}=
\frac{\lvert\sin(\pi\epsilon))\rvert+
\sqrt{\sin^2(\pi\epsilon)-\sin^2(\pi\ell)}}
{\lvert\sin(\pi\ell)\rvert}$\,.\vspace{1mm} On the other hand,
$\lVert T_\chi^{1,1}\rVert_{ub}^2=1+\frac2{\sin^2(\pi\ell)}
\bigl(\cos^2(\pi\ell)\sin^2(\pi\epsilon)+
\lvert\cos(\pi\ell)\sin(\pi\epsilon)\rvert
\sqrt{1-\cos^2(\pi\ell)\cos^2(\pi\epsilon)}\,\bigr)\vspace{.5mm}$
holds for all $-1<\ell<0\,,\;\epsilon\in\R$\,. In particular, for
$\lvert\sin(\pi\epsilon)\rvert=\lvert\sin(\pi\ell)\rvert$ (these are the
exceptional cases $\epsilon\in\pm\ell+\Z$ mentioned above) one
gets\vspace{.5mm} $\lVert T_\chi^{1,1}\rVert_{ub}^2=
1+2\cos^2(\pi\ell)+2\cos(\pi\ell)\sqrt{1+\cos^2(\pi\ell)}$\,\,.
\\[1mm plus1mm]
The estimates given in \cite{ACD} for $\lVert\cdot\rVert_{ub}$ in the case
of $\SL(2,\R)$ have analogues for~$\widetilde G$\,. There exist $c_1,c_2>0$
such that for optimal weights $(a_o,b_o)$ one has (slightly improving the
bounds of \cite{ACD}\;L.\,4.9 and also the weaker earlier bounds of
\cite{KS}\;p.\,38)
\;$c_1\,<\,\lVert T_\chi^{a_o,b_o}\rVert_{ub}\;\big/\;
\frac1{\lvert\ell_\R\rvert}\,(\lvert\im(\ell)\rvert+1)^{\ell_\R+\frac12}
\min(\lvert\ell\rvert+\lvert\epsilon\rvert,1)\,<\,c_2$ for all
$\lvert\epsilon\rvert\le\frac12$ and all $\ell\in\C$ with
$-\frac12\le\ell_\R<0$ \,and
\;$c_1\,<\,\lVert T_\chi^{a_o,b_o}\rVert_{ub}\;\big/\;
\frac1{\ell_\R+1}\,(\lvert\im(\ell)\rvert+1)^{-\ell_\R-\frac12}\cdot\linebreak
\min(\lvert\ell+1\rvert+\lvert\epsilon\rvert,1)\,<\,c_2$ for
$-1\!<\!\ell_\R\!\le\!\frac12$\,. Similarly for $\lVert T_\chi^{1,1}\rVert_{ub}$
\;(with a slightly bigger interval). The asymptotics of \cite{ACD}\;Th.\,1.1
for $\lvert\im(\ell)\rvert\to\infty$ carries over for every $\epsilon\in\R$\,.
\\[.3mm plus.5mm]
The representations $T_\chi^{a,b}$ act on different Hilbert spaces
$H^{\ell_\R+\frac12}_{a,b}$\,. Put
$(M_\ell f)\sphat(t)=\lvert t\rvert^{\ell+\frac12}\hat f(t)$\,. This defines
isometric isomorphisms $M_\ell\!:H^{\ell_\R+\frac12}_{1,1}\to L^2(\R)$
\,(correspond\-ing to the operators $W(s,\frac12)$ of \cite{KS}\;p.\,22 with
$s=\ell+1$\,, combined with Fourier duality). Then
$g\mapsto M_\ell\,T_\chi^{1,1}(g)M_\ell^*\ (-1<\ell_\R<0\,,\;\epsilon\in\R)$
is a family
of uniformly bounded representations of $\widetilde G$ on $L^2(\R)$\,.
Similarly as in \cite{KS}\;p.\,37 one can show that for fixed
$\epsilon\in\R\,,\,g\in\widetilde G$ \,the operators depend {\it analytically}
on $\ell$\,. For $\ell_\R=-\frac12$ one gets unitary equivalents of the
principal series representations. But for $\ell\in\R\,,\,\ell\neq-\frac12$
and $\epsilon\notin\frac12\Z$ the representations (being unitarily equivalent
to $T_\chi^{1,1}$) are just boundedly equivalent to the representations
$T_\chi^c$ from the complementary series. With a similar construction,
inserting additional multiplication operators, \cite{Sa}\;p.\,56 defines
isometric isomorphisms $M_\ell\!:H^{\ell_\R+\frac12}_{a_o,b_o}\to L^2(\R)$
(for optimal weights $(a_o,b_o)$ as above) and this gives another
family of uniformly bounded representations containing also (unitary
equivalents of) the complementary series representations. But for this family
the domain of analyticity (resp.\;definition) is smaller. For
$0<\lvert\epsilon\rvert\le\frac12$ one gets points of infinity at
$\ell=-\lvert\epsilon\rvert,\lvert\epsilon\rvert-1$ and jump discontinuities
on the segments
$\ell\in\,]\!-\!\lvert\epsilon\rvert,0\,[\;\cup\;
]\!-\!1,\lvert\epsilon\rvert-1[$\,,
see \cite{Sa}\;Th.\,3.0.1.
\vskip1mm plus1mm
\item[(b)] \ In \cite{W} Wiersma considered $L^p$-representations
(and corresponding $C^*$-algebras; see also \cite{BGu},\cite{BR}). In the
irreducible case, the $L^2$-representations (of a locally compact group)
are just the square integrable
representations (\cite{Wa}\;4.5.9) and the $L^1$-representations are
the integrable representations (as defined by Harish-Chandra,
\cite{Wa}\;p.\,356). For a square integrable irreducible representation
$\pi$\,, all coefficient functions belong to $L^2(G)$
\,(\cite{Wa}\;L.\,4.5.9.1), i.e. $A_\pi\subseteq L^2(G)$
\,(notation of \cite{W}\;sec.\,2).
For integrable irreducible representations (and more generally
$L^p$-representations with $p\neq2$) this need not be true: \ when
considered as representations of $\SL(2,\R)$ \,(i.e. $\ell\in-\frac12\N$),
looking at the growth of $\fr P_{mn}^\ell(x)$ described in
Remark\;\ref{remdisc}\,(a), it follows easily that $T_\ell^\pm$ are
integrable for $\ell<-1$\,. From Theorem\;\ref{apcodisc} and the modified
estimate of
\ref{remdisc}\,(b), one can show that for the coefficients $t_{mn}^\ell$
associated to $\Cal P_{mn}^\ell$ (see \ref{di31}), one has
$\lVert t_{mn}^\ell\rVert_1\to\infty$
for $m\to\infty\ (n$ fixed). Hence (using \cite{Wa}\;Prop.\,4.5.9.6 and
duality) there are non-integrable coefficients for $T_\ell^+$.
\\[2.5mm plus.5mm]
If the centre $Z$ of $G$ is not compact there are no irreducible
$L^p$-representations (for $p<\infty$). Similarly to \cite{MW}, one can
modify the definition using mixed-norm spaces. For a bounded continuous
function $f\!:G\to\C$ put
\,$f_{Z,\infty}(\dot x)=\sup_{z\in Z}\lvert f(xz)\rvert\
(x\in G\,,\,\dot x=xZ)$. We call $\pi$ an \,$L^p$-representation$\!\mod Z$\,,
if $f=\pi_{\eta\eta}$ satisfies $f_{Z,\infty}\in L^p(G/Z)$ for $\eta$ in
a dense subset of $\Cal H_\pi$\,.  Dually, for $Z$ discrete, we consider
\,$f_{Z,1}(\dot x)=\sum_{z\in Z}\lvert f(xz)\rvert\,\
(x\in G\,,\,\dot x=xZ$\,; for general $Z$ one uses a Haar measure of $Z$).
$L^p_{Z,1}(G)$ denotes the space of classes of Borel measurable functions $f$
for which $f_{Z,1}\in L^p(G/Z)$.
\\[1mm plus.5mm]
Considering auxiliary groups $G_\epsilon=(\widetilde G\times\T)/D_\epsilon$\,,
where $D_\epsilon=\{(k\pi,e^{-2\pi ik\epsilon}):k\in\Z\,\}$ and using
the $\lVert\ \rVert_{ub}$ estimates from (a) above, one can extend the
interpolation techniques of \cite{KS}\;Th.\,4 to the representations
$T_{(\ell,\epsilon)}^{1,1}$ for every real $\epsilon$\,. This gives analogues
of \cite{KS}\;Th.\,7 and its corollaries: \ for $1\le p<2\,,\;
T_{(\ell,\epsilon)}^{1,1}$ extends to
$L^p_{Z,1}(\widetilde G)$ when $-\frac1p<\re(\ell)<\frac1p-1$\,.
\\[1mm]
For the complementary series, it follows that for $-1<\ell\le-\frac12\,,\
\lvert\epsilon\rvert<1+\ell\,,\ T_{(\ell,\epsilon)}^c$ extends to
$L^p_{Z,1}(\widetilde G)$  when $1\le p<-\frac1\ell$\;. Dually,
$T_{(\ell,\epsilon)}^c$ is an $L^q$-representation$\!\mod Z$ for
$q>\frac1{\ell+1}$ and this holds in the stronger sense that
$f_{Z,1}\in L^q(\widetilde G/Z)$ is satisfied for all coefficients $f$\,. For
$q=\frac1{\ell+1}$ there are no non-zero coefficients with
$f_{Z,1}\in L^q(\widetilde G/Z)$ \,(using the asymptotics of
$\fr P_{mn}^\ell(x),\ x\to\infty$\,, given in \ref{di41}).
This corresponds to the bounds given in \cite{W}\;p.\,3948 for
$\SL(2,\R)$
(in the notation there $r=2\ell+1$\,). Recall that
$T_{(-1-\ell,\epsilon)}^c\cong T_{(\ell,\epsilon)}^c$\,.
\\[.5mm]
For the discrete series, it follows from the asymptotics given in
Remark\;\ref{remdisc}\,(a) that for any $\ell<0\,,\ T_\ell^\pm$ are
$L^q$-representations$\!\mod Z$ for all $q>-\frac1\ell$ and for
$-1\le\ell<0\,,\ q=-\frac1\ell$ \,there are no non-zero coefficients with
$f_{Z,1}\in L^q(\widetilde G/Z)$\,. In particular, for
$\ell<-\frac12\,,\ T_\ell^\pm$
are square integrable$\!\mod Z$ and (as mentioned above) it follows
that $T_\ell^\pm$ extend to $L^2_{Z,1}(\widetilde G)$ and that
$f_{Z,1}\in L^2(\widetilde G/Z)$ for all coefficients~$f$\,.
For any $q<2$ and $\ell<0$ (using again Theorem\;\ref{apcodisc} and the
modified estimate of \ref{remdisc}\,(b)\,) there are some coefficients
of $T_\ell^\pm$ for which $f_{Z,1}\notin L^q(\widetilde G/Z)$.
\\[1mm plus.5mm]
For the principal series, i.e. $\re\ell=-\frac12\,,\ \epsilon\in\R$\,,
$T_{(\ell,\epsilon)}$ extends to
$L^p_{Z,1}(\widetilde G)$  when $1\le p<2$\,, $T_{(\ell,\epsilon)}$ is an
$L^q$-representation$\!\mod Z$ (in the stronger sense) for $q>2$\,,
but there are no non-zero coefficients with
$f_{Z,1}\in L^2(\widetilde G/Z)$ \,(this property could also be obtained
from the ``principe de majoration", reducing to $T_{(-1/2,0)}$).
\\[1mm plus.5mm]
For finite covering groups of $\SL(2,\R)$ (or $\PSL(2,\R)$) and more
generally groups $G$ with compact centre $Z$ satifying $G/Z\cong\PSL(2,\R)$
(like the groups $G_\epsilon$ above) one can omit ``\!\!$\mod Z$\," from the
statements above, considering ordinary
$L^q$-representations and extensions to ordinary $L^p$-spaces for the
$\epsilon$ matching $G$.
\vskip2mm plus.5mm
\item[(c)] \ For a function $f\!:\R\setminus\{0\}\to\C$ put
$Jf(x)=e^{(1-\sgn(x))\pi i\epsilon}\lvert x\rvert^{2\ell}f(\frac1x)\
(x\neq0)$.
Then (similarly to \cite{VK}\,p.\,376 who uses in the case
$\epsilon=0,\frac12$ the notation $\hat f$ instead of $Jf$) one can
characterize the space $\Cal D_\chi^\R$ introduced in
Remark\;\ref{remrep}\,(a)
\,(as a domain for the smooth version $T_\chi^s$) as
\,$\Cal D_\chi^\R=\{f\in C^\infty(\R):\;Jf \text{ extends to a
$C^\infty$-function}\linebreak\text{on }\R\,\}=E_\chi C^\infty(\T)$\,.
This imposes additional conditions on
the asymptotic behaviour of $f$\,, e.g. %
$\lim_{x\to-\infty}\lvert x\rvert^{-2\ell}f(x)=
e^{-2\pi i\epsilon}\lim_{x\to\infty} x^{-2\ell}f(x)$ \,(and both limits have
to exist fi\-nitely). It contains the Schwartz-space $\Cal S(\R)$ \,and the
functions $e_m^\chi$ of \eqref{defbasis} (the eigenfunctions under the action
of $\widetilde K$).
$\Cal D_\chi^\R$ coincides with the space of $C^\infty$-vectors for the
representation $T_\chi$
(more precisely $T_\chi^\R$) on $\Cal H_\chi$ \,, the same for the
modifications $T_\chi^c$ on $\Cal H_\chi^c\ (-1<\ell<0)$ and
$T_\chi^{a,b}$ on $H^{\ell_\R+\frac12}_{a,b}$. For
$T_\ell^+\ (\ell<0)$ one gets the subspace of $f\in\Cal D_{(\ell,-\ell)}^\R$
that have analytic extensions to the upper half plane.
\\[3mm]
For the real model and $\epsilon\in\R\,,\,\ell\in\C$ one has
\,$\Cal H_\chi=E_\chi L^2(\T)=L^2(\R,\linebreak(x^2+1)^{-2\re(\ell)-1}dx)$
\,(for general $\epsilon\!\in\!\C$ it is
$L^2(\R,e^{-4\im(\epsilon)\arccot(x)}(x^2+1)^{-2\re(\ell)-1}dx)$).
\\[3mm plus1mm]
The classification of the representations $T_\chi$ can be done similarly as
in \cite{VK}\;6.4 (see also \cite{HT}\,Prop.\,II.1.2.9,\,p.\,63; recall that
the corresponding Harish-Chandra module of
$\widetilde K$-finite vectors is isomorphic to
$U(\nu^+,\nu^-)$ with $\nu^+=-\ell+\epsilon\,,\ \nu^-=-\ell-\epsilon$).
$T_\chi$ is indecomposable for all
$\chi\in\C^2\setminus\{(-\frac12,\epsilon)\!:\;\epsilon\in\frac12+\Z\}$.
It is (topologically) irreducible iff $\epsilon\notin\pm\ell+\Z$
\,(the same for the Harish-Chandra modules).
For $\epsilon\in\ell+\Z$ the closed subspace generated by
$e_n^\chi\,,\ n\le\ell-\epsilon$ is invariant  (this corresponds to
\cite{HT}\,Prop.\,II.1.2.9\,(c),(d); for $\ell\notin\frac12\,\N_0$ one
gets $\overline{V}_{2\ell}$ infinitesimally\,). Again, it depends only on the
coset of \,$\epsilon\!\!\mod1$\,. For $\epsilon=\ell$ it is given by
$E_\chi H^2(\T)$ as a generalized Hardy space of analytic functions on the
lower half plane. The induced representation
on the quotient space corresponds to $V_{-2\ell-2}$ (but for
$\re(\ell)\neq-\frac12$ the corresponding group representations are not
boundedly equivalent, see below).
This gives also a way to get, for real $\ell>-1$\,, $T_{-\ell-1}^+$
from a quotient of
$T_{(\ell,\ell)}$ (with a suitable renorming), similarly
with $T_{-\ell-1}^-$ from $T_{(\ell,-\ell)}$, see also \cite{VK}\,p.\,302.
For $\epsilon\in-\ell+\Z$ the closed subspace generated by
$e_n^\chi\,,\ n\ge-\ell-\epsilon$ is invariant (this corresponds to
\cite{HT}\,Prop.\,II.1.2.9\,(b),(d)\,).
For $\ell\in\frac12\,\Z\,,\,\epsilon\in\ell+\Z$ both
invariant subspaces are present. If $\ell<0$ they have trivial intersection.
For $\ell\ge0$ the intersection has dimension $2\ell+1$\,, giving all
finite dimensional representations of $\SL(2,\R)$. There are no
further closed invariant subspaces in $\Cal H_\chi$\,.
\\[2mm plus1mm]
$T_\chi$ depends only on the coset  of \,$\epsilon\!\!\mod1$\,. If
$\epsilon\notin\pm\ell+\Z$ then by \cite{VK}\;6.4.4 the representations
$T_{(\ell,\epsilon)}$ and $T_{(-\ell-1,\epsilon)}$ are infinitesimally
equivalent ($\ell\in\C$), but (using e.g. \cite{O}\;Ch.4,\,(5.02)\,)
they are not boundedly equivalent, unless $\re(\ell)=-\frac12$\,. For
$\re(\ell)=-\frac12\,,\;\epsilon\in\R$ they are unitarily equivalent;
for $\epsilon\notin\R$ they are just boundedly equivalent. There are
no further equivalences among the $T_\chi$ \,(in particular, $T_\ell^+$ and
$T_{-\ell-1}^+$ are not infinitesimally equivalent for $-1<\ell<0,\,
\ell\neq-\frac12$,
the same for $T_\ell^-$ and $T_{-\ell-1}^-$). As mentioned before,
$T_{(\ell,\epsilon)}^c$ and $T_{(-\ell-1,\epsilon)}^c$ are unitarily
equivalent. More generally, for the optimal weights $(a_o,b_o)$ (depending
on $\chi$), $T_{(\ell,\epsilon)}^{a_o,b_o}$ and
$T_{(-\ell-1,\epsilon)}^{a_o,b_o}$ are unitarily equivalent. But
(for $\epsilon\notin\{0,\frac12,\pm\ell\}+\Z$\,), $T_{(\ell,\epsilon)}^{1,1}$
and $T_{(-\ell-1,\epsilon)}^{1,1}$ are just boundedly equivalent.
It is not hard to see that $T_\chi$ is uniformly bounded on $Z$
iff $\epsilon\in\R$ and $T_\chi$ is uniformly bounded on $H$ iff
$\re(\ell)=-\frac12$. Combined, $T_\chi$ is uniformly bounded
iff $\epsilon\in\R$ and $\re(\ell)=-\frac12$, and then $T_\chi$ is even
unitary (see also \cite{Sa}\,p.\,37).
\vskip2mm
\item[(d)] \ Similar approximations can be done for $G=\SU(2)$ and the
functions $P^\ell_{mn}$ of \cite{VK}\;6.3, describing there
the matrix coefficients.
\end{Rems}
%=============================Literatur====================================

\end{document}